\documentclass[hidelinks, twoside]{article}
\usepackage{makecell}
\usepackage{xcolor}

\usepackage{subcaption}
\usepackage{hyperref}
\usepackage[ruled]{algorithm2e}
\usepackage{amsmath,amsthm,amssymb, mathtools}
\usepackage{cleveref}
\usepackage[accepted]{stylefile}

\newcommand{\RR}{\mathbb{R}}

\newcommand{\onelip}{1\text{\upshape-Lip}(\Omega)}

\newcommand{\lebesgue}{\mathcal{L}_d}

\DeclarePairedDelimiter\norm{\lVert}{\rVert}%
\DeclarePairedDelimiter\abs{\lvert}{\rvert}%

\DeclareMathOperator*{\essinf}{ess\,inf}
\DeclareMathOperator*{\spt}{spt}

\makeatletter
\let\oldabs\abs
\def\abs{\@ifstar{\oldabs}{\oldabs*}}
\let\oldnorm\norm
\def\norm{\@ifstar{\oldnorm}{\oldnorm*}}

\makeatother

\newtheorem{theorem}{Theorem}
\newtheorem{lemma}[theorem]{Lemma}

\newtheorem{proposition}[theorem]{Proposition}

\theoremstyle{definition}
\newtheorem{definition}[theorem]{Definition}

\theoremstyle{remark}
\newtheorem{remark}[theorem]{Remark}

\makeatletter
\def\ps@myheadings{%
    \let\@oddfoot\@empty\let\@evenfoot\@empty
    \def\@evenhead{\thepage\hfil\slshape\leftmark}%
    \def\@oddhead{{\slshape\rightmark}\hfil\thepage}%
    \let\@mkboth\@gobbletwo
    \let\sectionmark\@gobble
    \let\subsectionmark\@gobble
    }
  \if@titlepage
  \renewcommand\maketitle{\begin{titlepage}%
  \let\footnotesize\small
  \let\footnoterule\relax
  \let \footnote \thanks
  \null\vfil
  \vskip 60\p@
  \begin{center}%
    {\LARGE \@title \par}%
    \vskip 3em%
    {\large
     \lineskip .75em%
      \begin{tabular}[t]{c}%
        \@author
      \end{tabular}\par}%
      \vskip 1.5em%
    {\large \@date \par}
  \end{center}\par
  \@thanks
  \vfil\null
  \end{titlepage}%
  \setcounter{footnote}{0}%
}
\else
\renewcommand\maketitle{\par
  \begingroup
    \renewcommand\thefootnote{\@fnsymbol\c@footnote}%
    \def\@makefnmark{\rlap{\@textsuperscript{\normalfont\@thefnmark}}}%
    \long\def\@makefntext##1{\parindent 1em\noindent
            \hb@xt@1.8em{%
                \hss\@textsuperscript{\normalfont\@thefnmark}}##1}%
    \if@twocolumn
      \ifnum \col@number=\@ne
        \@maketitle
      \else
        \twocolumn[\@maketitle]%
      \fi
    \else
      \newpage
      \global\@topnum\z@   
      \@maketitle
    \fi
    \thispagestyle{plain}\@thanks
  \endgroup
  \setcounter{footnote}{0}%
}
\makeatother

\begin{document}

%
\runningtitle{A new method for determining Wasserstein 1 OT maps
from Kantorovich potentials, with deep learning applications}
\runningauthor{ Tristan Milne, {\'E}tienne Bilocq, and Adrian Nachman }

%

\twocolumn[

\conftitle{A new method for determining Wasserstein 1 optimal transport maps
from Kantorovich potentials, with deep learning applications}

\confauthor{ Tristan Milne\textsuperscript{1*} \And {\'E}tienne Bilocq\textsuperscript{1*} \And  Adrian Nachman\textsuperscript{1,2} }
\confaddress{University of Toronto}
\confaddress{\textsuperscript{1}Department of Mathematics \and \textsuperscript{2}The Edward S. Rogers Sr. Department of Electrical and Computer Engineering}
\vskip 0.3in plus 2fil minus 0.1in
]

\begin{abstract}
  Wasserstein 1 optimal transport maps provide a natural correspondence between points from two probability distributions, $\mu$ and $\nu$, which is useful in many applications.  Available algorithms for computing these maps do not appear to scale well to high dimensions. In deep learning applications, efficient algorithms  have been developed for approximating solutions of the dual problem, known as Kantorovich potentials, using neural networks  (e.g. \cite{gulrajani2017improved}).  Importantly, such algorithms work well in high dimensions. In this paper we present an approach towards computing Wasserstein 1 optimal transport maps that relies only on Kantorovich potentials. In general, a Wasserstein 1 optimal transport map is not unique and is not computable from a potential alone. Our main result is to prove that if $\mu$ has a density and $\nu$ is supported on a submanifold of codimension at least 2, an optimal transport map is unique and can be written explicitly in terms of a potential. These assumptions are natural in many image processing contexts and other applications. When the Kantorovich potential is only known approximately, our result motivates an iterative procedure wherein data is moved in optimal directions and with the correct average displacement. Since this provides an approach for transforming one distribution to another, it can be used as a multipurpose algorithm for various transport problems; we demonstrate through several proof of concept experiments that this algorithm successfully performs various imaging tasks, such as denoising, generation, translation and deblurring, which normally require specialized techniques.

\end{abstract}

\section{Introduction}
\label{sec:introduction}
Let $\mu$ and $\nu$ be probability distributions on a compact, convex domain $\Omega \subset \RR^d$. The Wasserstein 1 distance between $\mu$ and $\nu$, denoted $W_1(\mu,\nu)$, is given by
\begin{equation}
    W_1(\mu,\nu) = \inf_{T_{\#} \mu=\nu} \int_\Omega |x-T(x)| d\mu(x),\label{eq:wasserstein_1_def}
\end{equation}
where $T_{\#} \mu$ denotes the pushforward measure obtained by the formula $T_{\#} \mu(E) = \mu(T^{-1}(E))$. There are numerous applications of $W_1(\mu,\nu)$ in machine learning as a natural way of comparing distributions, notably for the training of Wasserstein GANs \cite{arjovsky2017wasserstein} where it serves as the objective function for training the generator. A solution $T_0$ to the optimization problem in \eqref{eq:wasserstein_1_def} --- called an optimal transport map --- is also of interest, since it  can be used to transport specific points $x$ distributed according to $\mu$ to naturally corresponding points $T_0(x)$ distributed according to $\nu$. This is useful in applications such as image denoising, translation, or deblurring where given $x$ sampled from $\mu$ we want a corresponding $y$ sampled from $\nu$ (i.e. a denoised, translated, or deblurred version). The transport map $T_0$ can also be used to sample $\nu$ in contexts where this correspondence is less important, such as image generation.

A commonly used method for approximating $W_1(\mu,\nu)$ for high dimensional problems with neural networks comes from Wasserstein GANs with Gradient Penalty (WGAN-GP) \cite{gulrajani2017improved}. This approach was inspired by the dual problem to \eqref{eq:wasserstein_1_def},
\begin{equation}
    W_1(\mu,\nu) = \sup_{|\nabla u| \leq 1} \int_\Omega u(x) d\mu(x) - \int_\Omega u(y) d\nu(y).\label{eq:dual_formula_for_W1}
\end{equation}
Here the constraint that $|\nabla u(x)| \leq 1$ almost everywhere is equivalent to $u \in \onelip$, the set of $1$-Lipschitz functions on $\Omega$. A solution $u_0$ to \eqref{eq:dual_formula_for_W1} is called a Kantorovich potential, and the method from \cite{gulrajani2017improved} is also often used to approximate $u_0$ (e.g. \cite{lunz2018adversarial}, \cite{tanaka2019discriminator}, \cite{ mohammadi2021regularization}). Obtaining a Wasserstein 1 optimal transport map $T_0$ is more challenging, however. Such a map has been proven to exist if $\mu$ has a density with respect to Lebesgue measure $\lebesgue$ (\cite{evans1999differential}, \cite{caffarelli2002constructing}, \cite{ambrosio2003existence}), but in general it is non-unique. Further, its construction requires information about $\mu$ and $\nu$ which is typically unavailable in practice, such as the cumulative distribution functions of conditional distributions of $\mu$ and $\nu$ on certain line segments of $\RR^d$ known as ``transport rays'' (see, e.g., \cite{caffarelli2002constructing}). On the other hand, some information on $T_0$ can be gleaned from a solution $u_0$ to \eqref{eq:dual_formula_for_W1}. Indeed, it is well known that such a function provides the normalized direction of transport, in the sense that $x \neq T_0(x)$ implies, with $\mu$ probability $1$, that
\begin{equation}
    -\nabla u_0(x) = \frac{T_0(x)-x}{|T_0(x) - x|}.\label{eq:normalized_direction_ofT}
\end{equation}
In the general setting this is as much as one can say about $T_0$ from $u_0$ alone. The main theoretical contribution of this paper is to provide conditions on $\mu$ and $\nu$ under which we \textbf{can} compute $T_0$ solely from a Wasserstein 1 Kantorovich potential. The key assumption is that, in addition to $\mu$ having a density with respect to Lebesgue measure in $\RR^d$ (a condition we denote by $\mu \ll \lebesgue$), $\nu$ is supported on a submanifold of $\RR^d$ of dimension no greater than $d-2$. This is natural in practice; for instance, if $\nu$ consists of real images, a common hypothesis in computer vision holds that it must inhabit a low-dimensional manifold in the ambient space $\RR^d$ \cite{pope2021intrinsic}. We will show that these assumptions on $\mu$ and $\nu$ induce a geometric condition on the Kantorovich potential $u_0$ such that one can determine the transport distance $|x-T_0(x)|$ using $u_0$ alone; since the direction of transport is already given by $-\nabla u_0(x)$, these two values specify the optimal transport map uniquely and provide a formula for it.

Our main theorem in this direction is the following. We note that in addition to constructing a transport map from a Wasserstein 1 Kantorovich potential, we prove that this map is unique. This is an extension of the uniqueness result of \cite{hartmann2020semi}, where it is assumed that $\nu$ is discrete and that $d \geq 2$; note that this case is subsumed by our hypotheses. In the statement of the theorem, $\spt(\nu)$ denotes the support of the measure $\nu$, which is the smallest closed set which $\nu$ assigns a measure of $1$.
\begin{theorem}
\label{thm:map_from_potential}
Let $\mu$ have a density with respect to Lebesgue measure $\lebesgue$, and let $\spt(\nu) \subset M$, where $M$ is a $C^1$ submanifold of $\RR^d$ with $\dim(M) \leq d-2$. Then the optimal transport map $T_0$ for $W_1(\mu,\nu)$ is unique up to modification on sets of $\mu$ measure zero. It is given by
\begin{equation}
    T_0(x) = x - \alpha(x) \nabla u_0(x),\label{eq:OT_map_from_potential_only}
\end{equation}
for $\alpha$ defined by
\begin{equation*}
    \alpha(x) = \sup\{ |x-z| \mid z\in \Omega, u_0(x) - u_0(z) = |x-z|\}.
\end{equation*}
\end{theorem}

In terms of applications, this result motivates an algorithm for approximating an optimal transport map $T_0$ for $W_1(\mu,\nu)$. A naive approach would be to compute $\alpha(x)$ and use \eqref{eq:OT_map_from_potential_only} to obtain $T_0(x)$, however this requires exact knowledge of a Kantorovich potential $u_0$ and the step size $\alpha$. Since in practice there is error in both $u_0$ and $\alpha$, we found it more effective to use a constant step size $\eta$ obtained by averaging $\alpha(x)$ with respect to $\mu$. In other words, we modify the distribution $\mu$ by pushing it forward under the map
\begin{equation*}
    T(x) = x- \eta \nabla u_0(x).
\end{equation*}
Our averaging argument dictates that $\eta$ should be given by the average value of $\alpha(x)$, which is simply $W_1(\mu,\nu)$ by \eqref{eq:OT_map_from_potential_only}, \eqref{eq:normalized_direction_ofT}, and \eqref{eq:wasserstein_1_def}. 
Fortuitously, the value of $W_1(\mu,\nu)$ is available to us at no extra computational cost as a by-product of computing $u_0$ via \eqref{eq:dual_formula_for_W1}. Additionally, we have a simple  theoretical result (\Cref{thm:guaranteed_reduction_if_lessthanmedian}) concerning gradient descent on a Kantorovich potential with a spatially uniform step size, giving a general condition under which it yields a decrease in the Wasserstein 1 distance. 

Naturally, the use of a spatially uniform step size introduces some new error in approximating $T_0$, as we may overshoot or undershoot our targets depending on the size of $|x-T_0(x)|$ relative to $\eta$.  To ameliorate this, we iterate this procedure, optionally learning a new Kantorovich potential at each stage to correct for past errors.

Since our implementation uses the method from \cite{gulrajani2017improved} to learn $u_0$, where approximate Kantorovich potentials are called ``critics'', we dub the iterative transport procedure that results from this step size selection method Trust the Critics (TTC)\footnote{Note that the TTC algorithm is a simplified and improved version of that of \cite{milne2021trust}}.

The main contributions of this paper are as follows:
\begin{enumerate}
    \item We obtain a novel theoretical result (\Cref{thm:map_from_potential}) showing that under reasonable assumptions on $\mu$ and $\nu$ an optimal transport map can be derived solely from a Wasserstein 1 Kantorovich potential. 
    As a corollary we obtain a new uniqueness result for optimal transport maps in this setting which generalizes a theorem of \cite{hartmann2020semi}
    \item Motivated by our theory, we devise a novel approximate transport algorithm (TTC). In addition to the connection with \Cref{thm:map_from_potential}, this approach is justified by \Cref{thm:guaranteed_reduction_if_lessthanmedian}, which provides an estimate on the reduction in the Wasserstein 1 distance obtained by a gradient descent step with a constant step size on a Kantorovich potential.
    \item We show that TTC works well in practice as a multipurpose algorithm by applying it to several high dimensional imaging problems (denoising, generation, translation, and deblurring), which typically require specialized approaches.
\end{enumerate}

The rest of this paper is structured as follows. In \Cref{sec:related} we discuss related work. In \Cref{sec:theoreticalresults} we provide necessary background on Wasserstein 1 optimal transport and sketch the proof of \Cref{thm:map_from_potential}. We also provide a statement of \Cref{thm:guaranteed_reduction_if_lessthanmedian} and a sketch of the proof; details of all proofs are deferred to \Cref{app:proofs}. In \Cref{sec:practicalalgorithms} we explain our implementation of TTC.  \Cref{sec:experiments} details our experiments, which include applications of TTC to a variety of high dimensional imaging problems; additional experiments and hyperparameter settings are given in \Cref{sec:exp_appendix}. Finally, we summarize the paper in \Cref{sec:conclusion}.







\section{Related work}
\label{sec:related}
To our knowledge, \Cref{thm:map_from_potential} is the first approach for computing an optimal transport map from a Wasserstein 1 Kantorovich potential alone. In classical results on the construction of a Wasserstein 1 optimal transport map (e.g. \cite{evans1999differential}, \cite{caffarelli2002constructing}, \cite{ambrosio2003existence}) the potential $u_0$ plays a key role, but additional information on the measures $\mu$ and $\nu$ is needed for the construction. Suppose that, for instance, $d=1$ and the support of $\nu$ is strictly to the right of the support of $\mu$. It is not difficult to show that the function $u_0(x) = -x$ is a Kantorovich potential for any such pair $(\mu,\nu)$, and consequently the potential alone does not suffice to produce a map. In this setting the cumulative distribution functions of $\mu$ and $\nu$ can be used to compute an optimal map, but such information is typically not available in many applications. Our geometric assumptions on $\mu$ and $\nu$ (see \Cref{thm:map_from_potential}) eliminate this issue.

The optimal transport map for the Wasserstein 2 distance can be obtained from a corresponding Kantorovich potential via a well known and simple formula (see, e.g., Theorem 1.17 of \cite{santambrogio2015optimal}). This formula is leveraged in several works, (e.g. \cite{lei2017geometric}, \cite{ makkuva2019optimal} \cite{korotin2020wasserstein}), however in this case the determination of the Kantorovich potentials requires the computation of a Legendre transform, which remains a challenge in high dimensions \cite{jacobs2020fast}. Our \Cref{thm:map_from_potential} provides an analogous result for the Wasserstein 1 distance, which benefits from having a dual problem (i.e. \eqref{eq:dual_formula_for_W1}) which is considerably simpler. Furthermore, the Wasserstein 1 distance has other advantages. For example, \cite{hartmann2020semi} points out that $W_1(\mu,\nu)$ behaves particularly well under affine transformations of $\mu$ and $\nu$ on the space of measures, which is relevant in imaging applications where these correspond to brightness or contrast adjustments.

There are relatively few methods for finding Wasserstein 1 optimal transport maps when $d>1$ if $\mu$ and $\nu$ are not both discrete measures, in contrast to the case of the Wasserstein 2 distance (e.g. \cite{benamou2000computational}, \cite{ angenent2003minimizing}). To our knowledge, the sole exception is \cite{hartmann2020semi}, which assumes that $\nu$ is discrete and does not include applications in dimensions higher than $2$. Thus, given the use of techniques such as that of \cite{gulrajani2017improved} for approximating Kantorovich potentials for large scale problems, we view \Cref{thm:map_from_potential} as a significant step for computing Wasserstein 1 optimal transport maps in high dimensions.

Since we often lack knowledge of an exact Kantorovich potential, we introduce TTC as an effective but approximate transport algorithm. In the context of specific applications, it is related to existing works; we will discuss examples of this for image generation and denoising. When $\mu$ is a noise distribution and $\nu$ governs a set of real data, TTC can be viewed as an image generation algorithm which extends the method of \cite{nitanda2018gradient}. In that paper the authors showed that fine-tuning of Wasserstein GANs can be accomplished by modifying generated data with gradient descent steps of constant step size on learned critics. In the generative context, the novel contribution of TTC is its  adaptive step size motivated by optimal transport theory; this yields much faster convergence of distributions which are not initially close. To make this work, we need a close approximation of the Wasserstein 1 distance. Motivated by \cite{milne2021wasserstein}, we obtain this approximation by using a much larger value of the regularization parameter $\lambda$ from WGAN-GP than is typical (see the discussion in \Cref{sec:practicalalgorithms}). 

If $\mu$ and $\nu$ are distributions of noisy and clean images, then TTC can be viewed as a denoising algorithm. In this context a single step of TTC is equivalent to the denoising method from \cite{lunz2018adversarial} with a particular regularization parameter (see \Cref{prop:advregis1ttc}). We demonstrate experimentally that improved performance is obtained by iterating this approach via TTC; note that an alternative iterative procedure is given in \cite{mukherjee2021end}.   

\section{Theoretical results}
\label{sec:theoreticalresults}
In this section we will sketch the proof of \Cref{thm:map_from_potential}, as well as state and sketch the proof of \Cref{thm:guaranteed_reduction_if_lessthanmedian}, which provides additional justification for TTC. Full details for the proofs can be found in \Cref{app:proofs}.

We begin with the proof of \Cref{thm:map_from_potential}, which requires some background from Wasserstein 1 optimal transport.
\subsection{Background on Wasserstein 1 optimal transport}
\label{sec:background_on_W1}
Central to our analysis is the concept of transport rays, a term coined in \cite{evans1999differential}, which refers to maximal segments over which the Lipschitz inequality of a $1$-Lipschitz function is saturated. Several definitions have been used in different works; ours is based on Definition 3.7 of \cite{santambrogio2015optimal}.
\begin{definition}
\label{def:transportray}
Let $u \in \onelip$. For $x, y \in \Omega$ the segment $[x,y] := \{ (1-t)x + ty \mid t \in [0,1]\}$ is called a transport ray of $u$ if $x \neq y$, $u(x) - u(y) = |x-y|$, and $[x,y]$ is not properly contained in any other segment $[z,w]$ satisfying these conditions. The open segment $]x,y[ := \{ (1-t)x + ty \mid t \in (0,1)\}$ is called the interior of the transport ray, and $x$ and $y$ are called its upper and lower endpoints, respectively.
\end{definition}
Transport rays play a key role in the Wasserstein 1 optimal transport problem since if $u_0$ is a Kantorovich potential for $W_1(\mu,\nu)$, its transport rays specify where optimal mass transport can occur. Specifically, we have that $\mu$ almost surely, $x\neq T_0(x)$ implies that $[x, T_0(x)]$ is contained in a transport ray. This is the content of the following well known result.
\begin{lemma}
\label{lem:ae_in_transport_rays}
Let $\mu \ll \mathcal{L}_d$. Suppose that $u_0$ and $T_0$ are a Kantorovich potential and optimal transport map for $W_1(\mu,\nu)$, respectively. Then $\mu$ almost everywhere,\label{lem:spt_T0_in_transport_ray}
\begin{equation*}
    u_0(x) - u_0(T_0(x)) = |x-T_0(x)|.
\end{equation*}
\end{lemma}

The existence of transport rays imposes additional structure on a  Kantorovich potential $u_0$. The following lemma states that $u_0$ is affine on transport rays and differentiable on their interiors, with gradient parallel to the ray. Note that this result, together with \Cref{lem:spt_T0_in_transport_ray}, implies the claim made in \Cref{sec:introduction} that $-\nabla u_0$ points in the direction of optimal mass transport.

\begin{lemma}[Essentially Lemmas 3.5 and 3.6 from \cite{santambrogio2015optimal}]
\label{lem:affineanddifferentiableoninterior}
If $[x,y]$ is a transport ray of $u$ then for all $t \in [0,1]$,
\begin{equation}
u((1-t)x + ty) = (1-t) u(x) + t u(y). \label{eq:uisaffineonrays}
\end{equation}
Further, $u$ is differentiable for all  $z \in ]x,y[$, with derivative satisfying
\begin{equation}
\nabla u(z) = \frac{x-y}{|x-y|}. \label{eq:graduisraydirection}
\end{equation}
\end{lemma}
A consequence of \Cref{lem:affineanddifferentiableoninterior} is that two transport rays can only intersect at a point which is an endpoint of both. It is also easy to prove that that point must be an upper or lower endpoint for both rays.
\begin{lemma}
\label{lem:two_rays_only_cross_at_endpoints}
    If two distinct transport rays $[x,y]$ and $[x',y']$ of a function $u$ intersect at a point $w$, then either $w = x = x'$ or $w = y = y'$.
\end{lemma}
Put another way, this lemma tells us that once a given transport ray collides with another, neither ray can continue. This basic notion forms a key part of the proof of \Cref{thm:map_from_potential}.

The second key notion is that away from the endpoints of transport rays, the function $x \mapsto \nabla u_0(x)$ is Lipschitz continuous. To state this result we must quantify the distance to the endpoints of a transport ray. We have already defined the distance to the lower endpoint with the function $\alpha$; the distance to the upper endpoint, which we denote by $\beta$, is defined  in the following proposition. The Lipschitz property of $\nabla u_0$ away from the ray endpoints was first proven in \cite{caffarelli2002constructing}, but there it is stated with sufficiently specialized notation that it may be helpful to provide a restatement here; see \Cref{app:proofs} for a proof. 

\begin{proposition}
\label{prop:Lipschitzgradient}
Let $u_0 \in \onelip$. Define $\alpha: \Omega \rightarrow \RR$ as in \Cref{thm:map_from_potential}, and $\beta:\Omega \rightarrow \RR$ as
\begin{equation*}
     \beta(x) = \sup\{ |x-z| \mid z\in \Omega, u_0(z) - u_0(x) = |x-z|\}.
\end{equation*}
For $j \in \mathbb{N}$, set
\begin{equation}
A_j = \{ z \in \Omega \mid \min(\alpha(z), \beta(z)) > 1/j\}.\label{eq:Ajdef}
\end{equation}
Then $z \mapsto \nabla u_0(z)$ is Lipschitz on $A_j$ with constant $4j$.
\end{proposition}
\subsection{Proof sketch for \Cref{thm:map_from_potential}}
\label{subsec:proofsketch_of_theorem1}
Using \Cref{prop:Lipschitzgradient} and \Cref{lem:two_rays_only_cross_at_endpoints}, we may sketch the proof of \Cref{thm:map_from_potential}. Intuitively, the proof holds because our assumptions on $\mu$ and $\nu$ force the transport rays of $u_0$ to focus on $\spt(\nu)$. Necessarily, this means that they collide with one another on $\spt(\nu)$ and therefore must end, and thus the transport distance $|x-T_0(x)|$ is precisely equal to $\alpha(x)$.
\begin{proof}[\Cref{thm:map_from_potential} proof sketch]
Let $T_0$ be an optimal transport map for $W_1(\mu,\nu)$, which exists by \cite{ambrosio2003existence}. We begin by proving that a point $x$ sampled from $\mu$ is within a transport ray of $u_0$ ending at $T_0(x)$ with probability one. Such a result will guarantee that $\mu$ almost surely, 
\begin{equation*}
    |x-T_0(x)| = \alpha(x),
\end{equation*}
whence \eqref{eq:OT_map_from_potential_only} will follow using \eqref{eq:normalized_direction_ofT}. To see that $x$ is in a transport ray of $u_0$ with $\mu$ probability $1$, we observe that $\mu(M) = 0$, and since $T_0(x) \in \spt(\nu)\subset M$ with $\mu$ probability $1$ we therefore have $x \neq T_0(x)$ $\mu$ almost everywhere. Thus \Cref{lem:spt_T0_in_transport_ray} implies that $x$ is in a transport ray of $u_0$ with $\mu$ probability $1$. Finally, to see that this transport ray ends at $T_0(x)$, we consider the set $A$ of $y \in \spt(\nu)$ that are not at the end of a transport ray, i.e.
\begin{equation}
    A = \{ y \in \spt(\nu) \mid \alpha(y) \beta(y) >0 \}.\label{def:Adef}
\end{equation}
We claim that $\nu(A) = 0$, which means that one of $\alpha(y)$ or $\beta(y)$ is zero $\nu$ almost surely; since almost all mass has to travel a non-zero distance to reach $\spt(\nu)$, $\beta(y) >0$ almost surely, and thus $\alpha(y) = 0$ almost surely, meaning that transport rays end with $\mu$ probability $1$ at $T_0(x)$.

To prove that $\nu(A) = 0$, we use the equation $\nu(A) = \mu(T_0^{-1}(A))$, and prove that $T_0^{-1}(A)$ is Lebesgue negligible and thus has $\mu$ measure $0$ since $\mu \ll \lebesgue$. By definition of $A$ and \Cref{prop:Lipschitzgradient}, we obtain that $T_0^{-1}(A)$ can be described with countably many Lipschitz coordinate systems of size $m+1$. Indeed, if $x\in T_0^{-1}(A)$, then $T_0(x) \in M$ and is in the interior of a unique transport ray. We can therefore write
\begin{equation}
    x = z + t \nabla u_0(z),\label{eq:lipschitz_parametrization}
\end{equation}
where $z = T_0(x) \in M$ and $t$ is a bounded parameter by compactness of $\Omega$. Since $z \in M$ we can specify it with $m$ Lipschitz coordinates, and by \Cref{prop:Lipschitzgradient} the pair $(z, t)$ is a Lipschitz parametrization of $x$ via \eqref{eq:lipschitz_parametrization}. Since $m+1 < d$, this shows that $\lebesgue(T^{-1}_0(A)) = 0$, as claimed. 

Our uniqueness result follows immediately from the representation formula \eqref{eq:OT_map_from_potential_only}, since we started with an arbitrary optimal transport map $T_0$.
\end{proof}
As we mentioned in \Cref{sec:introduction}, in applications we will replace the ideal step size $\alpha(x)$ with a uniform value $\eta$. This is because we typically do not know the Kantorovich potential $u_0$, and hence the ideal step size $\alpha$, precisely. The following simple result gives a condition under which such a gradient descent step decreases the Wasserstein 1 distance. This result is of general interest whenever gradient descent on a Kantorovich potential is used (e.g. see the description of the denoising method of \cite{lunz2018adversarial} given in \Cref{prop:advregis1ttc}). A more detailed version of this estimate is provided in \Cref{app:proofs}.
\begin{proposition}
\label{thm:guaranteed_reduction_if_lessthanmedian}
Let $\mu \ll \lebesgue$, and let $u_0$ and $T_0$ be a Kantorovich potential and optimal transport map for $W_1(\mu,\nu)$. Let $\tilde{\mu}$ be the pushforward of $\mu$ under one step of gradient descent on $u_0$, (i.e. $\tilde{\mu} = (I-\eta \nabla u_0)_\# \mu$ where $I$ is the identity map).  If $\eta >0$ and
\begin{equation*}
    \mu(\{ x\in \Omega \mid |x-T_0(x)| \geq \eta\}) > \frac{1}{2},
\end{equation*} then 
\begin{equation}
    W_1(\tilde{\mu},\nu) < W_1(\mu,\nu). \label{eq:guaranteed_reduction}
\end{equation}
\end{proposition}
\begin{remark}
\label{remark:measurability}
Note that the definition of the pushforward  $(I-\eta \nabla u_0)_\# \mu$ requires some care since $\nabla u_0$ only exists almost everywhere. This measure is well defined when $\mu\ll\lebesgue$; see \Cref{app:proofs} for details. 
\end{remark}

\begin{proof}[Proof sketch]
The set $\{x \in \Omega \mid |x-T_0(x)| \geq \eta\}$ is precisely the set of points where overshooting does not occur after applying the map $I-\eta \nabla u_0$. All these points $x$ move closer to their target $T_0(x)$ by distance $\eta$. The remaining points overshoot their targets by no more than distance $\eta$. Using these observations to estimate the transport cost from $\tilde{\mu}$ to $\nu$ yields \eqref{eq:guaranteed_reduction}.
\end{proof}

\section{Practical algorithms}
\label{sec:practicalalgorithms}
In this section we will give detailed descriptions of our implementation of TTC. Set $\mu_0:= \mu$. For $n \in \{1,2, \ldots, N\}$, assume $\mu_{n-1}\ll \lebesgue$,  and define
\begin{equation}
\mu_{n} = (I-\eta_{n-1} \nabla u_{n-1})_\# \mu_{n-1}\label{eq:mundef}.
\end{equation}
Here $u_{n-1}$ is a critic approximating a Kantorovich potential for $W_1(\mu_{n-1},\nu)$, and $\eta_{n-1}$ is an approximation of $W_1(\mu_{n-1},\nu)$; since $\mu_{n-1} \ll \lebesgue$ by assumption, we have that $\mu_n$ is well defined by \eqref{eq:mundef} following \Cref{remark:measurability}. Let us note that we must assume $\mu_{n-1} \ll \lebesgue$ since \Cref{thm:guaranteed_reduction_if_lessthanmedian} includes no guarantee that the property of having a density with respect to Lebesgue measure is preserved by a gradient descent step on a Kantorovich potential. In fact, such a result is quite challenging to prove if overshooting occurs, since $I-\eta \nabla u_0$ need not be an invertible map in this case. We view this as an interesting avenue for future work, and will proceed under the assumption that such a result holds; we note that the validity of this assumption does not appear to be an issue in practice.

To approximate the Kantorovich potentials $(u_n)_{n=0}^{N-1}$ we train standard critic neural networks from the literature using the one sided gradient penalty from \cite{gulrajani2017improved} (see \eqref{eq:grad_pen_def}), which was found to provide more stable training than the two sided penalty in \cite{petzka2018regularization}. In order to use this technique it is necessary to be able to sample both $\mu_n$ and $\nu$.  For $\mu_n$, we draw an initial point $x_0$ from $\mu_0$, and apply gradient descent maps from the sequence of pre-trained critics. Precisely, a sample $x \sim \mu_n$ is obtained via the formula
\begin{equation}
x = (I-\eta_{n-1}\nabla u_{n-1}) \circ \ldots \circ (I-\eta_0 \nabla u_0)(x_0).\label{eq:generator_formula}
\end{equation}
Regarding the computation of the step size $\eta_{n-1}$, we note again that the value of $W_1(\mu_n,\nu)$ is available as a by-product of computing a Kantorovich potential, so this choice of adaptive step size requires no extra computation in practice. More precisely, we use the negative of the minimal value of the functional from WGAN-GP \cite{gulrajani2017improved}, that is
\begin{equation}
W_1(\mu_n,\nu) \approx \frac{1}{M}\sum_{j=1}^M u_n(x_j) - u_n(y_j) - \lambda G(\nabla u_n(\tilde{x}_j)), \label{eq:w1approx}
\end{equation}
where the $x_j$ and $y_j$ are samples from $\mu_n$ and $\nu$ respectively, $M$ is the mini-batch size, $\tilde{x}_j$ is a random convex combination of $x_j$ and $y_j$ as in \cite{gulrajani2017improved} and
\begin{equation}
    G(\nabla u_n(z)) = (|\nabla u_n(z)|-1)_+^2,\label{eq:grad_pen_def}
\end{equation}
for $(a)_+ = \max(0,a)$. When training WGANs, researchers will typically use $\lambda = 10$ for the gradient penalty coefficient (e.g. \cite{gulrajani2017improved}, \cite{lunz2018adversarial}, \cite{mukherjee2021end}), however we use a value of $\lambda = 1000$. In practice we found that this value of $\lambda$ stabilizes the estimates of $W_1(\mu_n,\nu)$ and the training of TTC; using smaller values of $\lambda$ leads to inflated estimates of $W_1(\mu_n,\nu)$, leading to overly large step sizes and unstable training. This is confirmed by the analysis in \cite{milne2021wasserstein}, which shows that at best the value in \eqref{eq:w1approx}, in expectation, converges to $W_1(\mu_n,\nu)$ like $O(\lambda^{-1})$. Depending on the mini-batch size $M$ the value of \eqref{eq:w1approx} can vary considerably across mini-batches, so we compute an average over $100$ mini-batches after training is completed. 

We found in practice that when $W_1(\mu,\nu)$ is large, a significant acceleration can be obtained by reusing the same critic for several steps. Consequently we pre-select a set of indices $J\subset \{ 1, 2, \ldots \} $ where we train $u_n$ only if $n \in J$. Whether we train $u_n$ or not, we warm start its parameters by initializing them at those of the preceding critic $u_{n-1}$ when $n\geq 1$. The method for training TTC is summarized in Algorithm \ref{alg:TTC}. 
\begin{algorithm}
\SetAlgoLined
\KwData{Samples from source $\mu$ and target $\nu$, untrained critics $(u_n)_{n=0}^{N-1}$ with parameters $(w_n)_{n=0}^{N-1}$, gradient penalty coefficient $\lambda$, number of critic training iterations $C$, batch size $M$, indices of critics to train $J$, Adam parameters $(\epsilon_c, \beta_1,\beta_2)$.}
\KwResult{A distribution $\mu_N$ which can be sampled from $\mu$, $(u_n)_{n=0}^{N-1}, (\eta_n)_{n=0}^{N-1}$ via \eqref{eq:generator_formula}.}
\For{$n \in \{0, \ldots, N-1\}$}{
 \If{$n \in J$}{\For{$i\in \{0, \ldots, C-1\}$,}{
  	Sample minibatches $\{x_j\}_{j=1}^M$, $\{y_j\}_{j=1}^M$, and $\{t_j\}_{j=1}^M$ from $\mu_n$ (via \eqref{eq:generator_formula}), $\nu$, and $U([0,1])$\;
  		$\tilde{x}_j \leftarrow (1-t_j)x_j + t_jy_j$\;
  		$L_{i} \leftarrow \frac{1}{M}\displaystyle{\sum_{j=1}^M} u_n(y_j) - u_n(x_j) \newline
  		\hspace*{0.4 in} + \lambda G(\nabla u_n(\tilde{x}_j))$\;
  		$w_n \leftarrow \text{Adam}(L_i, \epsilon_c, \beta_1,\beta_2)$\;}}
With $M'=100M$, sample minibatches $\{x_j\}_{j=1}^{M'}$, $\{y_j\}_{j=1}^{M'}$, and $\{t_j\}_{j=1}^{M'}$ from $\mu_n$ (via \eqref{eq:generator_formula}), $\nu$, and $U([0,1])$\;
  $\eta_n \leftarrow \frac{1}{M'}\displaystyle{\sum_{j=1}^{M'}}u_n(x_j) - u_n(y_j) - \lambda G(\nabla u_n(\tilde{x}_j))$\;
  \lIf{$n < N-1$}{ $w_{n+1} \leftarrow w_n$}}
\caption{TTC Training}\label{alg:TTC}
\end{algorithm}

\section{Experiments}
\label{sec:experiments}

A large number of computational problems can be formulated as searching for a method to transform a ``source" probability distribution into a ``target" one in an optimal way. This is what is provided by the approximate transport map obtained from TTC. We demonstrate the versatility of TTC by applying it to four types of imaging problems: denoising, generation, translation and deblurring. For denoising, we compare TTC to an algorithm from \cite{lunz2018adversarial}, assessing the quality of the images obtained using PSNR. For generation, we compare TTC to WGAN-GP \cite{gulrajani2017improved}, evaluating the performance of both methods using the Fréchet Inception Distance (FID) \cite{heusel2017gans}. For translation and deblurring, we limit our contribution to proof of concept experiments and judge TTC based on qualitative results. The link to a GitHub repository containing our code is included in \Cref{sec:exp_appendix}. 

In the case of image generation, the specific pairings between source and target samples obtained by approximating a Wasserstein 1 transport map with TTC does not hold a special significance. This is because  we use a source distribution consisting of Gaussian noise and train TTC to transport it towards a target distribution from which only samples are known; we can then generate new images from the target by applying TTC to randomly sampled Gaussian noise images. By contrast, for the other three applications we consider, it is of crucial importance to preserve underlying pairings between source and target samples. When denoising an image, for example, it is obviously important that the result be a clean version of the same image. By providing an approximate transport map, TTC naturally finds a correspondence between individual source and target samples which is appropriate for the task at hand. It does so without having to rely on explicit dataset labels -- in this sense, TTC performs unsupervised learning.

\subsection{Image denoising}
\label{sec:denoising}

\begin{figure}
    \centering
    \includegraphics[width = 0.45 \textwidth, clip = True]{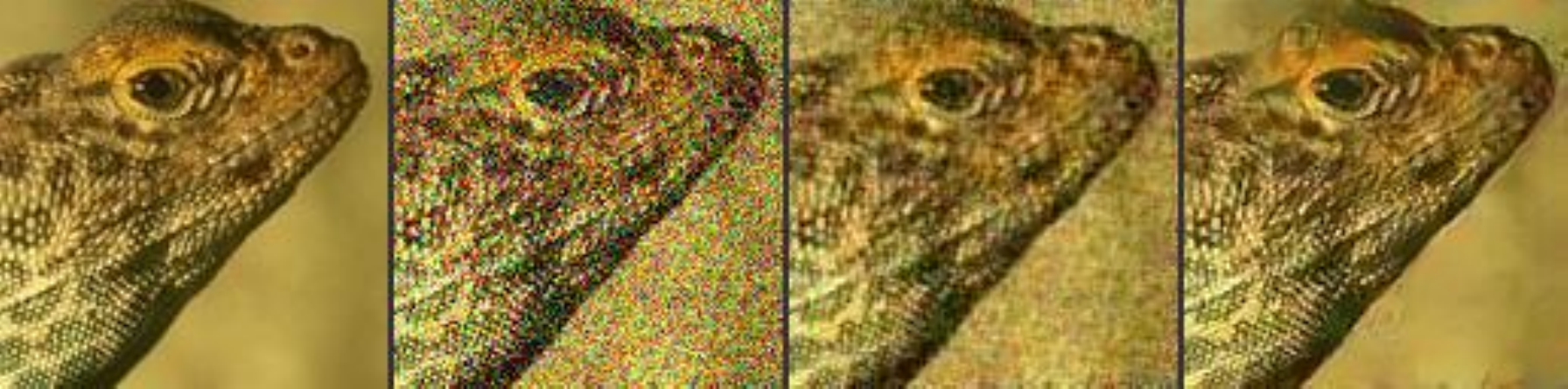}
    \caption{A single restored image from BSDS500 from noise level $\sigma = 0.2$. From left to right: original image, noisy image (PSNR = 14.0), restored image using \cite{lunz2018adversarial} (PSNR = 21.8), restored image using TTC (PSNR = 23.5).}
    \label{fig:denoising}
\end{figure}

\begin{table}
\centering
\begin{tabular}{c|ccc}
\multicolumn{1}{l}{} &
\multicolumn{3}{c}{Average PSNR (dB)} \\ 
 $\sigma$ & Noisy Image & Adv. Reg.& TTC  \\
\hline \\[-1.8ex]
0.1 & $20.0 \pm 0.03$ &  $27.1 \pm 0.9$ & $\mathbf{30.1} \pm 2.4$ \\
0.15 & $16.5 \pm 0.03$ & $24.8 \pm 0.8$ & $\mathbf{27.8} \pm 2.5$ \\
0.2 & $14.0 \pm 0.03$ & $23.0 \pm 0.9$ & $\mathbf{26.6} \pm 2.6$ \\
\end{tabular}
\caption{Results of the denoising experiments. PSNR values are reported as mean $\pm$ standard deviation, where the statistics are computed over the test set. In addition to having higher mean performance over the test set, TTC gives an improved PSNR for every image in the test set.}
\label{table:denoisingresults}
\end{table}

We apply TTC to restore images that have been corrupted with Gaussian noise. Specifically, we follow the experimental framework of \cite{lunz2018adversarial}, where the target distribution $\nu$ consists of random crops of the BSDS500 dataset \cite{arbelaez2010contour} and the source distribution $\mu$ is obtained from $\nu$ by adding i.i.d. Gaussian noise with standard deviation $\sigma$ to each image. In this setting, TTC bears an interesting relationship to the adversarial regularization method from \cite{lunz2018adversarial}. In that paper, a critic $u_0$ is obtained using the method from \cite{gulrajani2017improved} for  $W_1(\mu,\nu)$, and is then used as a learned regularizer in an inverse problem. This is applied to image restoration in the following way; given a noisy observation $x_0$, a denoised version is obtained by solving the minimization problem
\begin{equation}
\min_{x \in \Omega} \frac{1}{2}|x-x_0|^2 + \eta u_0(x),\label{prob:advreg}
\end{equation} 
where the parameter $\eta$ is estimated from the noise statistics. Incidentally, this requires the noise model to be known \textit{a priori}, as in \cite{moran2020noisier2noise}. In comparison, TTC does not require prior knowledge of the noise model because of its adaptive step size obtained by estimating $ W_1(\mu,\nu)$. The next proposition shows that, provided $\eta$ is small enough, the solution to \eqref{prob:advreg} is, in fact, equivalent to the solution obtained from a single step of TTC with step size $\eta$. As such, in this context, TTC can be thought of as an iterated form of the technique in \cite{lunz2018adversarial}, with an adaptive step size and where the critic is optionally updated after each reconstruction step.
\begin{proposition}
\label{prop:advregis1ttc}
Let $T_0$ be an optimal transport map for $W_1(\mu,\nu)$. If $\eta < \essinf_\mu |I-T_0|$, then for $\mu$-almost all $x_0$ there is a unique solution to \eqref{prob:advreg} given by\footnote{Recall, if $f:\Omega \rightarrow \RR$, then $\essinf_\mu(f)$ is the infimum of $f$ up to $\mu$ negligible sets, i.e. $\sup\{\ell \in \RR \mid \mu(f^{-1}((-\infty, \ell))) = 0\}$.}
\begin{equation}
x_1 = x_0 - \eta \nabla u_0(x_0).
\end{equation}
\end{proposition}
For a fair comparison of our results against the adversarial regularization denoising technique from \cite{lunz2018adversarial}, we use the critic architecture from that paper. Referring to Algorithm \ref{alg:TTC}, we train TTC with $N=20$ and $J = \{1, \ \dots \, 20\}$; unlike in the case of image generation (see Section \ref{sec:ttcvswgangp}), we have found that training the critic at each step was preferable for this application. The full list of hyperparameters used for denoising with TTC can be found in \Cref{table:ttc_hyperparameters}. We run denoising experiments at different noise levels given by the noise standard deviations $\sigma = 0.1$, $0.15$, $0.2$. Table \ref{table:denoisingresults} shows the averages and standard deviations of the PSNR values obtained using both methods on each image in a test dataset of $128 \times 128$ BSDS500 cropped images. TTC outperforms adversarial regularization on all the images in this test dataset. Figure \ref{fig:denoising} depicts the results of both algorithms on a specific test image. We note that, since we train a critic at each step, TTC is significantly more computationally demanding than adversarial regularization. Details of the computational resources used to train TTC for denoising can be found in \Cref{table:ttc_hyperparameters}. Though the PSNR values obtained with TTC are somewhat lower than the state of the art (e.g. \cite{moran2020noisier2noise}), we feel that our results are impressive given that we train on unpaired data, that we do not use prior knowledge of the noise model, and that our technique was not specifically designed for image denoising.

\subsection{Image generation}
\label{sec:ttcvswgangp}
\begin{figure}
    \centering
    \includegraphics[width = 0.45 \textwidth, clip = True]{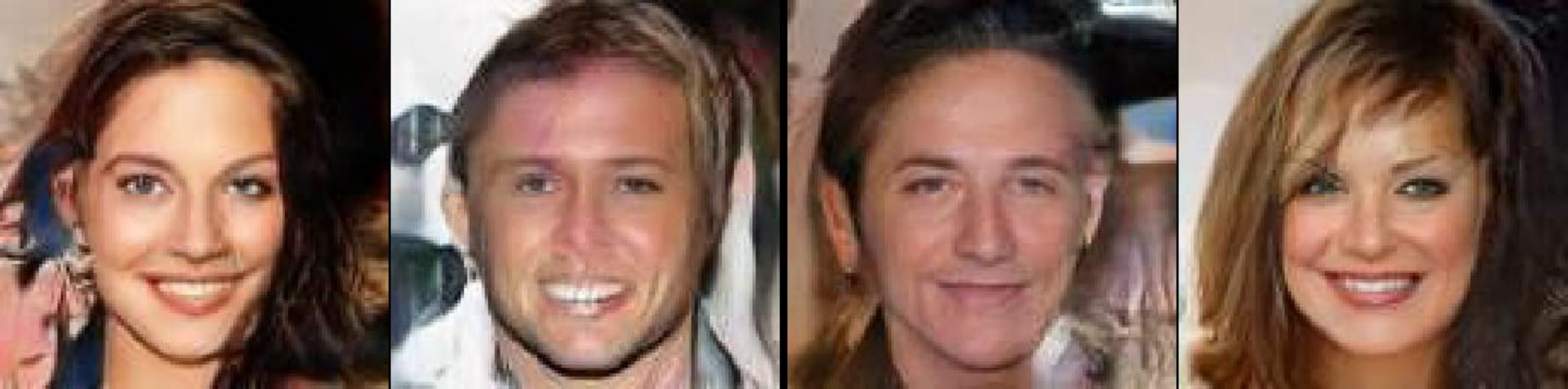}
    \caption{Generated samples produced by WGAN-GP and TTC trained on CelebaHQ. The two images on the left were generated with WGAN-GP and the two on the right with TTC. These images were some of the best that we could find within sets of 100 images produced with each algorithm while using the standard truncation trick on the noise input with a bound of 2.5.}\label{fig:generationexample}
\end{figure}

\begin{table}
\centering
\begin{tabular}{c|cc}
 & WGAN-GP & TTC  \\
\hline \\[-1.8ex]
MNIST & $18.3$ \hspace{0.07cm} \footnotesize{(200 min)}  & $\mathbf{6.5}$ \hspace{0.07cm} \footnotesize{(90 min)} \\
F-MNIST & $20.1$ \hspace{0.07cm} \footnotesize{(210 min)}  & $\mathbf{16.3}$ \hspace{0.07cm} \footnotesize{(90 min)} \\
CelebaHQ & $31.4$ \hspace{0.07cm} \footnotesize{(1290 min)}  & $\mathbf{31.2}$ \hspace{0.07cm} \footnotesize{(2030 min)} \\
\end{tabular}
\caption{Best FIDs obtained over the course of training for WGAN-GP and TTC, along with training time necessary to obtain this performance. These results were obtained without using the truncation trick. All experiments were run using one NVIDIA GPU; a P100 for MNIST and F-MNIST, and a V100 for CelebaHQ. TTC produces better FID values after a shorter training time than WGAN-GP for all datasets.}\label{table:TTCvsWGANGP}
\end{table}

We perform image generation experiments with TTC on three datasets: MNIST \cite{lecun1998gradient}, FashionMNIST \cite{xiao2017fashion} (abbreviated here as F-MNIST) and CelebaHQ \cite{karras2018progressive}. We evaluate TTC's generative performance through FID, using the implementation \cite{Seitzer2020FID}, and compare it to the performance we obtain using WGAN-GP \cite{gulrajani2017improved}. The MNIST and F-MNIST experiments are run at a resolution of $32 \times 32$ pixels with a single color channel, and use the InfoGAN architecture \cite{chen2016infogan} for the TTC critics as well as for the WGAN generator and critic. For TTC, we use $N=40$ and specify $J$ by training every other critic for the first 20 steps and then training every critic for the remaining 20 steps. The CelebaHQ experiment is run at a resolution of $128 \times 128$ pixels with three color channels, and uses the SNDCGAN architecture of \cite{kurach2019large} for all networks. For TTC, we take $N = 45$ and again specify $J$ by training every other critic for the first 20 steps and then training every critic for the remaining 25 steps. In all cases, we train both algorithms until their performance stops improving and we keep track of the training time required to reach this optimal state. As mentioned in Section \ref{sec:practicalalgorithms}, we use a gradient penalty parameter of $\lambda=1000$ for TTC, as this allows for a much better approximation of the Wasserstein 1 distance. When training WGAN-GP we use the standard value of $\lambda = 10$. Otherwise, the same hyperparameters are used for both algorithms whenever possible, e.g. mini-batch size, critic learning rate and Adam optimizer parameters. An important exception to this is the generator learning rate, which has a significant impact on WGAN-GP performance. We optimize this parameter for each dataset via a grid search. Details on all the training and FID evaluation parameters used for each experiment can be found in \Cref{table:ttc_hyperparameters}. The results are summarized in Table \ref{table:TTCvsWGANGP}, which includes the best FIDs attained and the training time required to reach them. TTC significantly outperforms WGAN-GP on the MNIST and F-MNIST datasets and requires a much shorter training time. The optimal performance of TTC and WGAN-GP are nearly equal on CelebaHQ, and while it is attained faster with WGAN-GP, we note that only very minor FID improvements occurred with TTC past the 1000 minute mark. Graphs presenting the relationship between training time and FID for both techniques and each dataset are included in \Cref{fig:fid_vs_time}.

\subsection{Image translation and deblurring}
\label{sec:translation}


\begin{figure}
    \centering
    \includegraphics[width = 0.45 \textwidth]{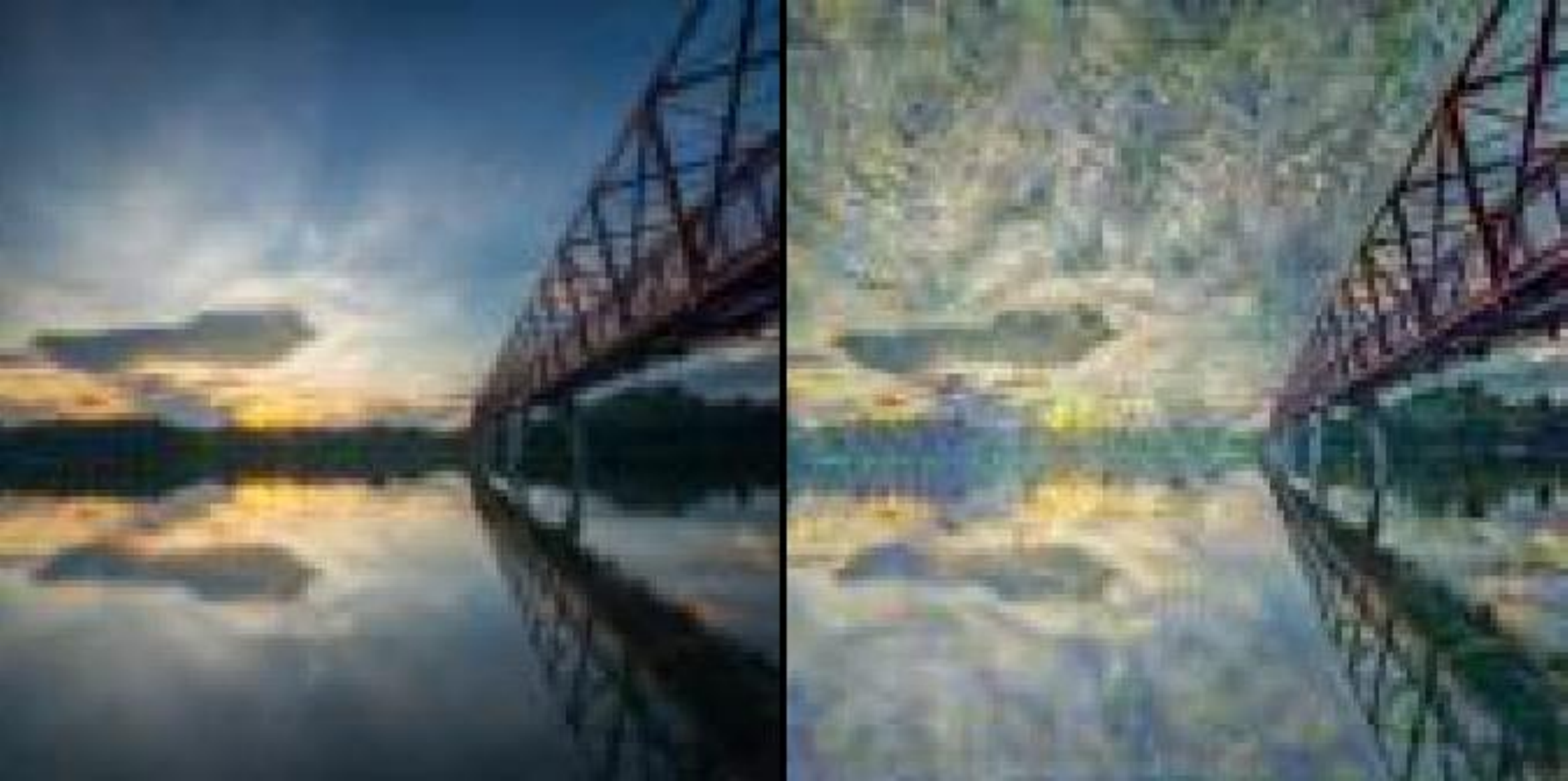}
    \caption{An example of TTC applied to translating landscape photos into Monet paintings.}\label{fig:translationexample}
\end{figure}
\begin{figure}
    \centering
    \includegraphics[width = 0.45 \textwidth, clip = True]{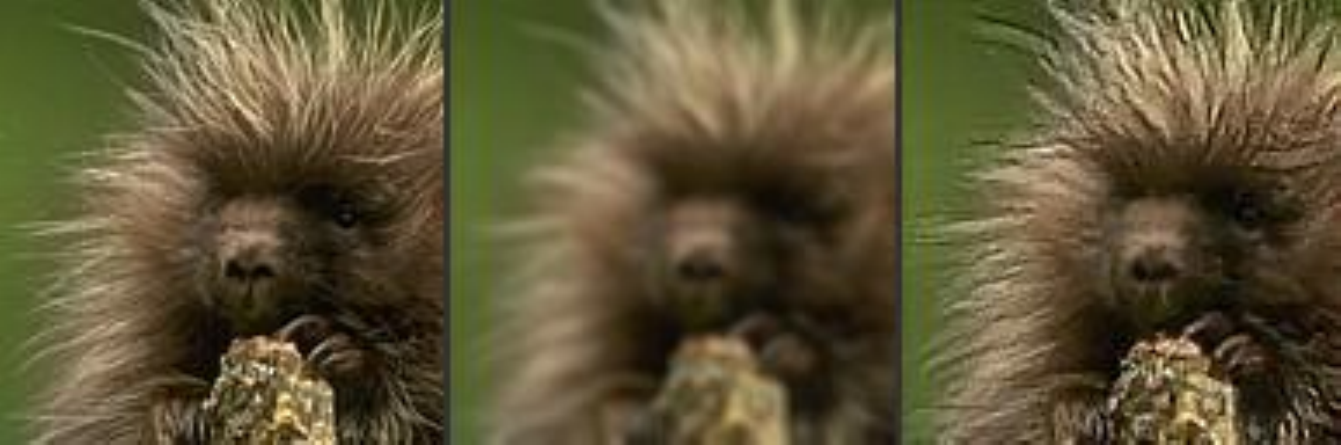}
    \caption{An example of TTC applied to deblurring. From left to right: original image, blurred image, restored image obtained with TTC.}\label{fig:deblurringexample}
\end{figure}

We further demonstrate the multipurpose flexibility of TTC by applying it to two additional problems: translation and deblurring. For translation, we use the Photograph and Monet datasets from \cite{zhu2017unpaired}. The former is used as the source, with the full images being resized to $128 \times 128$ pixels. We create the target distribution by taking $128 \times 128$ random crops of images in the Monet dataset. This means that we train TTC to translate real world images into corresponding “paintings” in the style of Monet; Figure \ref{fig:translationexample} shows an example of this. For deblurring, we once more use $128 \times 128$ random crops of images from the BSDS500 dataset. We create the blurred source distribution by using a $5 \times 5$ random Gaussian blurring filter with a standard deviation of $\sigma=2$. Figure \ref{fig:deblurringexample} displays an example of deblurring with TTC. For both of these experiments, we use the SNDCGAN critic architecture, and we train TTC with $N=20$ and $J = \{1, \ \dots \ , 20\}$. Details on the resources and time required for training are included in \Cref{table:ttc_hyperparameters}.

\section{Conclusion}
\label{sec:conclusion}
In this paper we have obtained a new formula for computing a Wasserstein 1 optimal transport map from a Kantorovich potential alone (\Cref{thm:map_from_potential}). This result holds if $\mu$ has a density and if $\nu$ is supported on a submanifold of codimension of at least $2$. Since these assumptions are natural in imaging problems as well as other applications, this result enables the computation of optimal transport maps for many problems of interest when a Kantorovich potential is known. For applications to high dimensional problems where the Kantorovich potential is only computed approximately, we proposed TTC, an iterative transport algorithm. This algorithm takes spatially uniform step sizes of the correct average displacement, and optionally trains new critics at each step to correct for the errors this introduces. The use of a spatially uniform step size was partly justified by \Cref{thm:guaranteed_reduction_if_lessthanmedian}. We also demonstrated through a variety of proof of concept experiments that TTC can be used as a multipurpose algorithm for various imaging tasks. This includes image denoising, generation, translation and deblurring, which normally require specialized approaches.

\bibliography{bibliography}


\onecolumn
\apptitle{A new method for determining Wasserstein 1 optimal transport maps from Kantorovich potentials, with deep learning applications --- Appendix}
\section{Detailed proofs}
\label{app:proofs}
\subsection{Proofs of background results for Wasserstein 1 transport}
In this section we include some proofs of the known results we stated in \Cref{sec:background_on_W1}, for the convenience of the reader. For the proofs of our novel results, the reader can skip to \Cref{sec:proofofmaintheorem}.

\begin{proof}[Proof of \Cref{lem:ae_in_transport_rays}]
Define
\begin{equation}
C = \{ x \in \Omega \mid  u_0(x) - u_0(T_0(x)) = |x-T_0(x)|\}. \label{eq:Cdef}
\end{equation}
Note that $C$ is Borel since $u_0$ is continuous and $T_0$ is Borel. If we can show $\mu(C) = 1$, we are done.  In fact, we can show that $\mu(\Omega \setminus C) = 0$ by the following argument: it is a standard consequence of the Kantorovich-Rubinstein formula (see e.g. the discussion following equation (3.2) in \cite{santambrogio2015optimal}) that
\begin{equation}
\spt((I,T_0)_\# \mu) \subset \{ (x,y) \in \Omega^2 \mid u_0(x) - u_0(y) = |x-y|\}. \label{eq:sptcontainedintrays}
\end{equation}
Letting $\Gamma_{T_0}(\Omega \setminus C) = \{ (x,T_0(x)) \mid x \in \Omega \setminus C\}$, \eqref{eq:sptcontainedintrays} gives us that
\begin{equation*}
(I, T_0)_\# \mu (\Gamma_{T_0}(\Omega \setminus C)) = 0,
\end{equation*}
but $(I, T_0)^{-1}(\Gamma_{T_0}(\Omega \setminus C)) = \Omega \setminus C$, so $\mu(\Omega \setminus C) = 0$.
\end{proof}
\begin{remark}
\label{remark:modifications_of_OT_map}
Note that we can redefine $T_0$ on $\mu$ negligible sets without affecting its optimality. Thus, by setting $T_0(x) = x$ for $x \in \Omega \setminus C$, we obtain that $u_0(x) - u_0(T_0(x)) = |x-T_0(x)|$ for all $x \in \Omega$. In the rest of this section we will often work with such $T_0$. 
\end{remark}
We will not include a proof of \Cref{lem:affineanddifferentiableoninterior} as it is essentially Lemmas 3.5 and 3.6 from \cite{santambrogio2015optimal}. These results can be partially extended; if $u$ is differentiable at the ray endpoints the same formula for the derivative from \Cref{lem:affineanddifferentiableoninterior} applies to these points. This is the content of the following result. 
\begin{lemma}
If $[x,y]$ is a transport ray of $u$ and $u$ is differentiable at either endpoint then \eqref{eq:graduisraydirection} also holds at that endpoint.\label{lem:graduisraydirectionextends}
\end{lemma}
\begin{proof}
The proof is contained in the proof of Corollary 3.8 from \cite{santambrogio2015optimal}.
\end{proof}

As an easy application of \Cref{lem:ae_in_transport_rays}, \Cref{lem:affineanddifferentiableoninterior}, and \Cref{lem:graduisraydirectionextends}, we can prove that $-\nabla u_0(x)$ gives the direction of optimal transport whenever $x \neq T_0(x)$. We stated this result informally in \Cref{sec:introduction}, and provide a formal statement and proof here. 
\begin{lemma}
\label{lem:graduistransportdirection}
Let $\mu \ll \mathcal{L}_d$. Suppose that $u_0$ and $T_0$ are a Kantorovich potential and optimal map, respectively, for $W_1(\mu,\nu)$. If $\mu\left(\{x \in \Omega \mid x \neq T_0(x)\} \right)>0$, then
\begin{equation}
    \mu(\{-\nabla u_0(x) = \frac{T_0(x)-x}{|T_0(x)-x|}\} \mid \{x \neq T_0(x)\}) = 1,\label{eq:conditional}
\end{equation}
where $\mu(A \mid B)$ is the conditional probability of event $A$ given $B$. In this sense, $x_0 \neq T_0(x)$ implies $-\nabla u_0(x) = \frac{T_0(x)-x}{|T_0(x)-x|}$ with $\mu$ probability $1$. 
\end{lemma}

\begin{proof}
Following \Cref{remark:modifications_of_OT_map},  we may assume with no loss of generality that
\begin{equation}
\{ x \in \Omega \mid u_0(x) - u_0(T_0(x)) = |x-T_0(x)|\} = \Omega. \label{eq:T_0_saturates_lip}
\end{equation} 
Since $u_0$ is Lipschitz, there is a Borel set $D \subset \Omega$ such that $u_0$ is differentiable on $D$ and $\mathcal{L}_d(\Omega \setminus D) = 0$.
We claim that 
\begin{equation}
    \{x \in \Omega \mid x \neq T_0(x) \text{ and } -\nabla u_0(x) \neq \frac{T_0(x)-x}{|T_0(x)-x|}\} \cap D = \emptyset.\label{eq:both_and_diff_is_empty}
\end{equation}
Indeed, suppose $x \neq T_0(x)$ and $x \in D$. By \eqref{eq:T_0_saturates_lip} we have that the segment $[x,T_0(x)]$ is in a transport ray. By \Cref{lem:affineanddifferentiableoninterior} or  \Cref{lem:graduisraydirectionextends}, we obtain that 
\begin{equation*}
\nabla u_0(x) = \frac{x-T_0(x)}{|x-T_0(x)|}.
\end{equation*}
As such, \eqref{eq:both_and_diff_is_empty} holds. Further, $\mu(\Omega \setminus D) = 0$ since $\mathcal{L}_d(\Omega \setminus D) = 0$, and $\mu \ll \mathcal{L}_d$. This, with \eqref{eq:both_and_diff_is_empty}, implies 
\begin{equation*}
\mu\left(\{x \in \Omega \mid x \neq T_0(x) \text{ and } -\nabla u_0(x) \neq \frac{T_0(x)-x}{|T_0(x)-x|}\}\right) = 0.
\end{equation*}
To prove \eqref{eq:conditional}, we have, by definition,
\begin{align*}
     \mu(\{-\nabla u_0(x) = \frac{T_0(x)-x}{|T_0(x)-x|}\} \mid \{x \neq T_0(x)\})&= \frac{\mu\left(\{x \in \Omega \mid x \neq T_0(x) \text{ and } -\nabla u_0(x) = \frac{T_0(x)-x}{|T_0(x)-x|}\}\right)}{\mu\left(\{x \in \Omega \mid x \neq T_0(x)\}\right)},\\
     &= \frac{\mu\left(x \in \Omega \mid x \neq T_0(x)\right)}{\mu\left(x \in \Omega \mid x \neq T_0(x)\right)},\\
     &= 1,
\end{align*}
as claimed.
\end{proof}

We can also use \Cref{lem:affineanddifferentiableoninterior} and \Cref{lem:graduisraydirectionextends} to prove \Cref{lem:two_rays_only_cross_at_endpoints}.
\begin{proof}[Proof of \Cref{lem:two_rays_only_cross_at_endpoints}]
Suppose without loss of generality that $w\in ]x,y[$. Then $\nabla u(w)$ exists by \Cref{lem:affineanddifferentiableoninterior} and
\begin{equation*}
    \nabla u (w) = \frac{x-y}{|x-y|}.
\end{equation*}
But since $u$ is differentiable at $w$ we also have via \Cref{lem:affineanddifferentiableoninterior} or \Cref{lem:graduisraydirectionextends} that
\begin{equation*}
    \nabla u (w) = \frac{x'-y'}{|x'-y'|},
\end{equation*}
which is a contradiction since $[x,y]$ and $[x',y']$ are distinct. So the crossing point $w$ must be an endpoint of both rays. Suppose that $w = x = y'$. Then
\begin{align*}
    u(x') - u(y) &= u(x') - u(y') + u(x) - u(y),\\
    &= |x'-y'| + |x-y|,\\
    &\geq |x'-y|.
\end{align*}
Since $u \in \onelip$, however, we have $u(x') - u(y) \leq |x'-y|$, and thus $u(x') - u(y) = |x'-y|$, and the preceding inequality is an equality. Thus the four points $x, y, x', y'$ are colinear, and $u$ saturates its Lipschitz bound on the segment $[x,y']$. This segment strictly contains the transport ray $[x,y]$, a contradiction to the definition of transport rays.
\end{proof}
Next we prove \Cref{prop:Lipschitzgradient}, which establishes that $x \mapsto \nabla u_0(x)$ is a Lipschitz function away from the endpoints of transport rays. The following proof is part of the proof of a larger result (Lemma 22) from \cite{caffarelli2002constructing}, but we include it here in a self contained form for the convenience of the reader.
\begin{proof}[Proof of \Cref{prop:Lipschitzgradient}]
Let $z, z' \in A_j$. Note that since $u_0 \in \onelip$, if $|z-z'| \geq \frac{1}{2j}$ then we have the trivial Lipschitz bound
\begin{equation}
|\nabla u_0(z) - \nabla u_0(z')| \leq 2 \leq 4 j |z-z'|.\label{eq:triviallipbound}
\end{equation}
Thus, we focus on the case $|z-z'| < \frac{1}{2j}$. In this case the Lipschitz constant of $u_0$ allows us to bound the variation in $u_0$ on these points;
\begin{equation}
|u_0(z) - u_0(z')| < \frac{1}{2j}.\label{eq:boundonuvariation}
\end{equation}
Set $w' = z' + (u_0(z) - u_0(z')) \nabla u_0(z')$. By \eqref{eq:boundonuvariation}, we have that $w'$ and $z'$ are on the same transport ray. Indeed, $w'$ is at most $\frac{1}{2j}$ away from $z'$, and $z'$ is at least $\frac{1}{j}$ from the endpoints of the transport ray it is contained in by definition of $A_j$.

Since $w'$ and $z'$ are on the same transport ray, $w'$ lies on the same level set of $u_0$ as $z$. Indeed, using  \Cref{lem:affineanddifferentiableoninterior},
\begin{align*}
u_0(w') &= u_0(z' + (u_0(z) - u_0(z'))\nabla u_0(z')),\\
&= u_0(z') + (u_0(z) - u_0(z')),\\
&= u_0(z).
\end{align*}
Since both $w'$ and $z$ are interior points of their transport rays and are on the same level set of $u_0$ we can then invoke Lemma 16 from \cite{caffarelli2002constructing} to obtain that
\begin{equation*}
|\nabla u_0(w') - \nabla u_0(z) | \leq \frac{1}{\sigma}|w'-z|,
\end{equation*}
where $\sigma$ is the minimal distance from $w'$ or $z$ to the endpoints of its transport ray; by construction this is at least $\frac{1}{2j}$. Hence,
\begin{equation*}
|\nabla u_0(w') - \nabla u_0(z) | \leq 2j|w'-z|.
\end{equation*}
Given that $\nabla u_0(w') = \nabla u_0(z')$, we therefore have
\begin{equation}
|\nabla u_0(z') - \nabla u_0(z) | \leq 2j|z'-z| + 2j|w'-z'|.\label{eq:preliminlipbound}
\end{equation}
Estimating the last term,
\begin{align*}
|w'-z'| &=  |u_0(z) - u_0(z')| ,\\
&\leq |z-z'|,
\end{align*}
whence \eqref{eq:preliminlipbound} gives us
\begin{equation*}
|\nabla u_0(z') - \nabla u_0(z) | \leq 4j |z'-z|
\end{equation*}
for all $|z'-z| \leq \frac{1}{2j}$. Combining this with \eqref{eq:triviallipbound}, we obtain that $z \mapsto \nabla u_0 
(z)$ is Lipschitz on $A_j$ with constant $4j$. 
\end{proof}
Since we will need it for the proof of \Cref{thm:map_from_potential}, we include here a result on the upper semi-continuity of $\alpha$ and $\beta$. The proof is almost identical to that of Lemma 24 in \cite{caffarelli2002constructing}. Indeed, the only difference is that our $\alpha$ and $\beta$ are defined using the supremum over the compact set $\Omega$, as opposed to $\spt(\nu)$ and $\spt(\mu)$, respectively. But compactness is the only essential ingredient in the proof of Lemma 24 from \cite{caffarelli2002constructing}, so there is practically no change to the argument. 

\begin{lemma}
If $\Omega$ is compact, the functions $\alpha$ and $\beta$ are upper semi-continuous.\label{lem:ell+uppersemicont}
\end{lemma}

\begin{proof}
We will only prove the result for $\beta$, as the result for $\alpha$ is quite similar. Let $(x_n)_n$ be a sequence in $\Omega$ such that $\lim_{n\rightarrow \infty} x_n = x_0$ and  $\lim_{n\rightarrow \infty} \beta(x_n) = \beta_0$. We seek to prove that
\begin{equation*}
\beta(x_0) \geq \beta_0.
\end{equation*}
We note that since $\Omega$ is compact, $\beta(x) < \infty$ for all $x$.  By definition, for each $n$ there exists $z_n \in \Omega$ such that
\begin{equation*}
\beta(x_n) - \frac{1}{n}\leq |z_n-x_n| = u_0(z_n) - u_0(x_n).
\end{equation*}
Since $\Omega$ is compact, a subsequence of the $z_n$ converges to some $z_0 \in \Omega$. Taking the limit of the preceding equation under this subsequence, and using continuity of $u_0$,
\begin{equation*}
\beta_0 \leq |z_0 - x_0| = u_0(z_0) - u_0(x_0).
\end{equation*}
As such, $|x_0-z_0| \leq \beta(x_0)$, establishing the desired inequality and proving that $\beta$ is upper semi-continuous.
\end{proof}

\subsection{Proof of \Cref{thm:map_from_potential}}
\label{sec:proofofmaintheorem}
In this section we will prove \Cref{thm:map_from_potential}, following the outline provided in \Cref{subsec:proofsketch_of_theorem1}. Throughout we will often tacitly assume the hypotheses of \Cref{thm:map_from_potential}, even though they may not all be needed in each of the results we prove here. First, we show that provided the set $A$ (see \eqref{def:Adef}) has $\nu$ measure zero, then $\alpha$ is equal to the transport length $|x-T_0(x)|$.
\begin{lemma}
\label{lem:alpha_is_transport_distance}
Let $u_0$ and $T_0$ be a Kantorovich potential and optimal transport map, respectively, for $W_1(\mu,\nu)$. If $\nu(A) = 0$, then $\mu$ almost everywhere,
\begin{equation}
    \alpha(x) = |x-T_0(x)|.\label{eq:alpha_is_transport_distance}
\end{equation}
\end{lemma}
\begin{proof}
Again, we may assume without loss of generality that
\begin{equation}
\{ x \in \Omega \mid u_0(x) - u_0(T_0(x)) = |x-T_0(x)|\} = \Omega. \label{eq:T_0_saturates_lip_2}
\end{equation} 
By definition, we obtain $\alpha(x) \geq |x-T_0(x)|$ for all $x$. To prove $\alpha(x) \leq |x-T_0(x)|$ $\mu$ almost surely, define $E$ as the set
\begin{equation}
 E = \{ x \in \Omega \mid |x-T_0(x)| < \alpha(x)\}.
\end{equation}
$E$ is Borel since $T_0$ is Borel and $\alpha$ is upper semi-continuous (\Cref{lem:ell+uppersemicont}). We aim to show that $\mu(E) = 0$. Since $\mu(T_0^{-1}(\Omega \setminus \spt(\nu))) = \nu(\Omega \setminus \spt(\nu))= 0$, we obtain that $E \cap T_0^{-1}(\Omega \setminus \spt(\nu))$ is $\mu$ negligible. Further, $\mu(\spt(\nu)) \leq \mu(M) = 0$, and there exists a Borel set $D$ such that $u_0$ is differentiable on $D$ and $\mathcal{L}_d(\Omega \setminus D) = 0$. Thus, to prove $\mu(E) = 0$, we need only show
\begin{equation*}
\mu(E \cap T_0^{-1}(\spt(\nu)) \cap (\Omega \setminus \spt(\nu)) \cap D) = 0.
\end{equation*}
For all $x$ in this set, $x \neq T_0(x)$. As such, the segment $[x, T_0(x)]$ is contained in a transport ray of $u_0$; because $\nabla u_0(x)$ exists, this is the unique transport ray that $x$ is in. Since $\alpha(x) > |x-T_0(x)|$, we have $\alpha(T_0(x)) > 0$. Further, since $T_0(x) \neq x$, $\beta(T_0(x)) >0$ as well. Thus,
\begin{equation}
E \cap T_0^{-1}(\spt(\nu))  \cap (\Omega \setminus \spt(\nu)) \cap D \subset T_0^{-1}(A)
\end{equation} 
which means that
\begin{equation}
\mu(E) \leq \mu(T_0^{-1}(A)) = \nu(A) = 0.
\end{equation}
Thus, for $\mu$ almost all $x$, $|x-T_0(x)| \geq \alpha(x)$, implying \eqref{eq:alpha_is_transport_distance}.
\end{proof}
Next we prove that $\mu$ having a density with respect to $\lebesgue$ and $\spt(\nu) \subset M$ implies that $\nu(A) = 0$.
\begin{proposition}
Under the assumptions of \Cref{thm:map_from_potential}, $\nu(A) = 0$.\label{prop:goodclassofproblems}
\end{proposition}
\begin{proof}
We will show $\nu(A) = \mu(T_0^{-1}(A)) = 0$ by showing that $T_0^{-1}(A)$ is contained in a set of Lebesgue measure $0$. For $j \in \mathbb{N}$, define
\begin{equation}
M_j := \spt(\nu) \cap A_j.
\end{equation}
recalling the set $A_j$ from \eqref{eq:Ajdef}. It is clear that
\begin{equation*}
A = \bigcup_{j=1}^\infty M_j,
\end{equation*}
and as a result if we show that $\mu(T_0^{-1}(M_j)) = 0$ for all $j$ we will be done. To prove this we will show that $T_0^{-1}(M_j)$ is contained in the image of a Lipschitz map from a Euclidean space with dimension strictly smaller than $d$.

To begin constructing this map, we first observe that via \Cref{prop:Lipschitzgradient} the map $y \mapsto \nabla u_0(y)$ is Lipschitz continuous on $M_j$. Since $M$ is a $C^1$ submanifold of $\RR^d$, there exists an atlas $\{ (U_i, \varphi_i) \}_{i=1}^\infty$, where $\varphi_i: U_i \rightarrow \RR^m$ with $\varphi_i^{-1}$ Lipschitz on $\varphi_i(U_i)$. Set $d_0 = \sup\{|x-y| \mid x, y \in \Omega\}$ and for $j \in \mathbb{N}$ define
\begin{equation*}
\tilde{V}_{ij} = \varphi_i(U_i \cap M_j), \quad V_{ij} = \{ (x,t) \in \RR^m \times \RR \mid  x \in \tilde{V}_{ij}, |t| \leq d_0 \}.
\end{equation*}
Let $\psi_{ij} : V_{ij} \rightarrow \RR^n$ be defined by
\begin{equation*}
\psi_{ij}(x,t) = \varphi_i^{-1}(x) + t \nabla u_0 (\varphi_i^{-1}(x)).
\end{equation*}
The map $\psi_{ij}$ is Lipschitz on $V_{ij}$ since $\varphi_i^{-1}$ is Lipschitz, $\nabla u_0$ is Lipschitz on $M_j$, and $|t|$ is bounded. By the Kirzbraun Theorem we may then extend $\psi_{ij}$ to a Lipschitz map on $\RR^{m+1}$.

We will now prove that
\begin{equation}
T_0^{-1}(M_j) \subset \bigcup_{i=1}^\infty \psi_{ij}(\RR^{m+1}). \label{eq:inclusiontoprove}
\end{equation}
Let $x \in T_0^{-1}(M_j)$. Then $T_0(x) \in M_j \subset M$, so there exists a chart $(U_i, \varphi_i)$ such that $T_0(x) \in U_i\cap M_j$. As such, there exists $z \in \tilde{V}_{ij}$ such that $T_0(x) = \varphi_i^{-1}(z)$. Moreover, since $T_0(x) \in M_j$, we have that $T_0(x)$ is on the interior of a unique transport ray of $u_0$, and via \eqref{eq:T_0_saturates_lip_2} we obtain that $x$ is on the same ray. Since $T_0(x)$ is on the interior of this ray, \Cref{lem:affineanddifferentiableoninterior} shows that $u_0$ is differentiable at $T_0(x)$ with derivative satisfying
\begin{equation*}
    \nabla u_0(T_0(x)) = \frac{x-T_0(x)}{|x-T_0(x)|}
\end{equation*}
provided $x \neq T_0(x)$. Thus, even if $x = T_0(x)$, there exists $t$ with $|t| \leq d_0$ such that
\begin{equation*}
x = T_0(x) + t \nabla u(T_0(x)) = \psi_{ij}(z,t).
\end{equation*}
This shows that $x \in \psi_{ij}(\RR^{m+1})$ for some $i$, and we therefore conclude that \eqref{eq:inclusiontoprove} holds. Note that due to Lipschitz property of $\psi_{ij}$ and the fact that $m+1 < d$,
\begin{equation}
\mathcal{L}_d(\psi_{ij}(\RR^{m+1})) = 0. \label{eq:Lebesgueiszero}
\end{equation}
This is a standard fact (c.f. Proposition 262D \cite{fremlin2000measure}). This confirms that $\mu(T_0^{-1}(A)) = 0$, since
\begin{align*}
\mu(T_0^{-1}(A)) &\leq \sum_{j=1}^\infty \mu(T_0^{-1}(M_j)),\\
&\leq \sum_{j=1}^\infty \sum_{i=1}^\infty \mu(\psi_{ij}(\RR^{m+1})),\\
&= 0,
\end{align*}
where the last line holds via \eqref{eq:Lebesgueiszero} and because $\mu \ll \mathcal{L}_d$.
\end{proof}
Using \Cref{lem:alpha_is_transport_distance} and \Cref{prop:goodclassofproblems} we can now prove \Cref{thm:map_from_potential}.
\begin{proof}[Proof of \Cref{thm:map_from_potential}]
Let $T_0$ be an optimal transport map for $W_1(\mu,\nu)$, which exists when $\mu \ll \lebesgue$ by \cite{ambrosio2003existence}. By \Cref{prop:goodclassofproblems} we have $\nu(A)= 0$, and hence by \Cref{lem:alpha_is_transport_distance} we obtain $\alpha(x) = |x-T_0(x)|$ $\mu$ almost everywhere. Since $x \neq T_0(x)$ $\mu$ almost everywhere, \Cref{lem:graduistransportdirection} implies that, $\mu$ almost surely,
\begin{equation*}
    x - \alpha(x) \nabla u_0(x) = x + |x - T_0(x)| \frac{T_0(x) - x}{|T_0(x) - x|} = T_0(x),
\end{equation*}
which is \eqref{eq:OT_map_from_potential_only}. Since we started with an arbitrary optimal transport map $T_0$ and showed via \eqref{eq:OT_map_from_potential_only} that it is expressible up to $\mu$ negligible sets only in terms of $u_0$, we obtain that $T_0$ is unique up to sets of $\mu$ measure $0$.
\end{proof}
\subsection{Proof of \Cref{thm:guaranteed_reduction_if_lessthanmedian}}
In this section we provide a proof of \Cref{thm:guaranteed_reduction_if_lessthanmedian}. First, however, we will clarify the meaning of the measure $(I-\eta\nabla u_0)_\# \mu$ which we hinted at in \Cref{remark:measurability}.

Recall that the claim is that $(I-\eta \nabla u_0)_\# \mu$ is well defined when $\mu \ll \lebesgue$. First we clarify why this deserves special attention. Since $u_0$ is only Lipschitz, $I-\eta \nabla u_0$ may not be a Borel map. Since $\mu$ is only a Borel measure, the standard definition of the pushforward may not be applicable because the pre-image of a Borel set may be only Lebesgue measurable and thus incompatible with $\mu$. This issue can be easily resolved, however, when $\mu \ll \lebesgue$. Since $I-\eta \nabla u_0$ is measurable, there is a Borel map $f_0$ almost everywhere equal to $I-\eta \nabla u_0$. The pushforward $(f_0)_\# \mu$ is well defined, and since $\mu \ll \lebesgue$ it is independent of the particular choice of $f_0$; this measure is what we mean when we write $(I-\eta \nabla u_0)_\# \mu$. Further, we may obtain samples in practice from $(f_0)_\# \mu$ by sampling $x \sim \mu$ and applying the map $I-\eta \nabla u_0$.

The following simple result provides a more detailed estimate than \Cref{thm:guaranteed_reduction_if_lessthanmedian}, which it proves as an immediate corollary.
\begin{proposition}
\label{thm:detailedguaranteed_reduction_if_lessthanmedian}
Let $\mu \ll \lebesgue$, and let $u_0$ and $T_0$ be a Kantorovich potential and optimal transport map for $W_1(\mu,\nu)$. Let $\tilde{\mu} = (I-\eta \nabla u_0)_\# \mu$. Then
\begin{equation}
 W_1(\tilde{\mu},\nu) \leq W_1(\mu,\nu) - \eta(2\mu(\{x \mid |x-T_0(x)| \geq \eta\}) - 1)  \label{eq:reduction_estimate}
\end{equation} 
In particular, if $\mu(\{ x\in \Omega \mid |x-T_0(x)| \geq \eta\}) > \frac{1}{2}$ and $\eta>0$, then 
\begin{equation*}
    W_1(\tilde{\mu},\nu) < W_1(\mu,\nu).
\end{equation*}
\end{proposition}
\begin{proof}
Let $f_0$ be a Borel map almost everywhere equal to $I-\eta \nabla u_0$. Then the measure $(f_0, T_0)_\#\mu$ is an admissible transport plan for $W_1(\tilde{\mu},\nu)$. By the definition of $W_1(\tilde{\mu},\nu)$, (see Section 5.1 of \cite{santambrogio2015optimal} for details), we have 
\begin{equation*}
    W_1(\tilde{\mu},\nu) \leq \int_\Omega |f_0(x) - T_0(x)| d\mu(x).
\end{equation*}
Estimating this integral, we obtain
\begin{align*}
    W_1(\tilde{\mu},\nu) &\leq \int_{|x-T_0(x)|< \eta}|f_0(x) - T_0(x)| d\mu(x) + \int_{|x-T_0(x)|\geq \eta}|f_0(x) - T_0(x)| d\mu(x),\\
    &= \int_{|x-T_0(x)|<\eta}|x-\eta \nabla u_0(x) -T_0(x)|d\mu(x) + \int_{|x-T_0(x)|\geq \eta}|x-\eta\nabla u_0(x) - T_0(x)| d\mu(x),\\
    &\leq \int_{|x-T_0(x)|<\eta}|x- T_0(x)|d\mu(x) + \eta \mu(\{x\mid |x-T_0(x)|<\eta\}) \\
    &\quad + \int_{|x-T_0(x)|\geq \eta}|x- T_0(x)| d\mu(x) - \eta \mu(\{x\mid |x-T_0(x)|\geq\eta\}),\\
    &= W_1(\mu,\nu) - \eta(2\mu(\{x \mid |x-T_0(x)| \geq \eta\}) - 1), 
\end{align*}
which is \eqref{eq:reduction_estimate}. Note that in the second inequality we have used the fact that $u_0 \in \onelip$ and \Cref{lem:graduistransportdirection}.
\end{proof}
\subsection{Proof of \Cref{prop:advregis1ttc}}
In this section we will prove \Cref{prop:advregis1ttc}, which links the denoising method from \cite{lunz2018adversarial} and gradient descent on a Kantorovich potential $u_0$.
\begin{proof}[Proof of \Cref{prop:advregis1ttc}]
Assume $\eta>0$; the result is trivial if $\eta=0$. Since $u_0 \in \onelip$, the minimal value of \eqref{prob:advreg} is bounded below by
\begin{equation}
    \min_{x\in \RR^d} \frac{1}{2}|x-x_0|^2 - \eta|x-x_0| + \eta u_0(x_0) = \eta u_0(x_0) - \frac{1}{2}\eta^2.
\end{equation}
The equality above follows by minimizing the one-dimensional function $z \mapsto \frac{1}{2}z^2 - \eta z$ over non-negative $z$, which has minimizer $z = \eta$. By assumption, for $\mu$ almost all $x_0$ we have $|x_0-T_0(x_0)| \geq \eta$, and the segment $[x_0, T_0(x_0)]$ is contained in a transport ray of $u_0$. In addition, for $\mu$ almost all $x_0$  \Cref{lem:graduistransportdirection} gives us that $|\nabla u_0(x_0)| = 1$. Thus, by \Cref{lem:affineanddifferentiableoninterior}, we get
\begin{align}
    \frac{1}{2}|\eta \nabla u_0(x_0)|^2 + \eta u_0(x_0 - \eta \nabla u_0(x_0)) &= \eta u_0(x_0) - \frac{1}{2}\eta^2.
\end{align}
Noting that $x_0 - \eta \nabla u_0(x_0) \in \Omega$ by convexity of $\Omega$, we get that $x_0 - \eta \nabla u_0(x_0)$ obtains the minimal value of \eqref{prob:advreg}. For uniqueness, observe that any minimizer $x^*$ distinct from $x_0 - \eta \nabla u_0(x_0)$ must not be equal to $x_0$ and satisfies
\begin{equation}
    u_0(x^*) = u_0(x_0) - |x^*-x_0|.
\end{equation}
Thus $x_0$ must exist at the intersection of at least two transport rays. The set of $x_0$ for which this can occur is negligible since it is contained in the set where $u_0$ is not differentiable, completing the proof.  
\end{proof}

\section{Experimental settings and additional results}
\label{sec:exp_appendix}

The code we created to run all the experiments presented in Section \ref{sec:experiments} is available on Github. Click \textcolor{blue}{\href{https://github.com/bilocq/Trust-the-Critics-2}{here}} to access the code used for the TTC experiments as well as the benchmark denoising experiments with adversarial regularization. Click \textcolor{blue}{\href{https://github.com/bilocq/wgan-gp-benchmark}{here}} to access the code for the benchmark generation experiments with WGAN-GP.

\subsection{Hyperparameters and computational resources}
\label{subsec:hyperparameters_appendix}
The hyperparameters and computational resources we used for all the experiments with our TTC algorithm are listed in Table \ref{table:ttc_hyperparameters}. Complementary information for the benchmark generation experiments with WGAN-GP is included in Table \ref{table:wgan_hyperparameters}. For all the generation experiments with TTC and WGAN-GP, we evaluated FID by comparing the full test datasets to either 10000 (MNIST and F-MNIST) or 3000 (CelebaHQ) generated samples -- these sample sizes match the sizes of the corresponding test datasets. For all of the denoising experiments with TTC and adversarial regularization, we evaluated PSNR separately on 200 test images and reported the mean and standard deviations of the results in Table \ref{table:denoisingresults}. For each of the benchmark denoising experiments with the adversarial regularization technique from \cite{lunz2018adversarial} -- corresponding to noise levels $\sigma = 0.1$, $0.15$, $0.2$ -- we used the first critic trained for the corresponding TTC denoising experiment, which took approximately 10 minutes to train, along with 200 gradient descent steps and a step size parameter of $0.05$ to solve (\ref{prob:advreg}).

\begin{table*}[h!]
\centering
\begin{tabular}{c|ccccc}
  & \makecell{\textbf{Denoising} \\ (All noise levels)} & \makecell{\textbf{Generation} \\ (MNIST / \\ F-MNIST)} & \makecell{\textbf{Generation} \\ (CelebaHQ)} & \textbf{Translation} & \textbf{Deblurring} \\
  \hline \\[-1ex]
  
  \makecell{$N$ \\ (Number of steps)} & 20 & 40 & 45 & 20 & 20  \\
  \hline \\[-1ex]
  
  \makecell{$J$ \\ (Steps where \\ critic is trained)} & \small{$\{0, 1, \ \dots \ , 19 \}$} & \makecell{\small{$\{0, 2, \ \dots \ , 18 \}$} \\ \small{$\cup \ \{19, 20, \ \dots \ , 39 \}$ }}  & \makecell{\small{$\{0, 2, \ \dots \ , 18 \}$} \\ \small{$\cup \ \{19, 20, \ \dots \ , 44 \}$ }} & \small{$\{0, 1, \ \dots \ , 19 \}$} & \small{$\{0, 1, \ \dots \ , 19 \}$}  \\
  \hline \\[-1ex]
  
  \makecell{$M$ \\ (Minibatch size)} & 32 & 128 & 32 & 32 & 32 \\
  \hline \\[-1ex]
  
  \makecell{$C$ \\ (Training iterations \\ per critic)} & 2500 & 1000 & 2500 & 2500 & 2500 \\
  \hline \\[-1ex]

  \makecell{$\lambda$ \\ (Gradient penalty \\ weight)} & 1000 & 1000 & 1000 & 1000 & 1000 \\
  \hline \\[-1ex]

  Learning rate & $0.0001$ & $0.0001$ & $0.0001$ & $0.0001$ & $0.0001$ \\
  \hline \\[-1ex]

  \makecell{Beta parameters \\ (Adam optimizer)} & $(0.5, \ 0.999)$ & $(0.5, \ 0.999)$ & $(0.5, \ 0.999)$ & $(0.5, \ 0.999)$ & $(0.5, \ 0.999)$ \\
  \hline \\[-1ex]

  Architecture & AR-Net & InfoGAN & SNDCGAN & SNDCGAN & SNDCGAN \\
  \hline \\[-1ex]
  
  \makecell{Image size \\ (Color channels $\times$ \\ height $\times$ width)} & \small{$3 \times 128 \times 128$} & \small{$1 \times 32 \times 32$} & \small{$3 \times 128 \times 128$} & \small{$3 \times 128 \times 128$} & \small{$3 \times 128 \times 128$} \\
  \hline \\[-1ex]

  \makecell{Training \\ dataset size} & 200 -- BSDS & 50000 & 27000  & \makecell{ 6287 -- Photo \\ 1072 -- Monet} & 200 -- BSDS \\
  \hline \\[-1ex] 

  \makecell{Testing \\ dataset size} & \makecell{200 -- BSDS \\ (PSNR)} & \makecell{10000 \\ (FID)} & \makecell{3000 \\ (FID)} & -- & -- \\
  \hline \\[-1ex] 

  \makecell{Total \\ training time \\ (minutes)} & 330 & 90 & 2030 & 630 & 740 \\
  \hline \\[-1ex] 

  GPU type & V100 & P100 & V100 & V100 & V100 \\
\end{tabular}
\caption{Hyperparameters and computational resources used for all TTC experiments, with notation referring to Algorithm \ref{alg:TTC}. We refer to the convolutional neural network architecture used for adversarial regularization in \cite{lunz2018adversarial} as AR-Net. The BSDS500 training and testing datasets used for the denoising and deblurring experiments contain 200 images each, but we applied data augmentation by taking random $128 \times 128$ crops of the images in these datasets. We did the same for the Monet dataset -- but not the Photo dataset -- in the translation experiment.}
\label{table:ttc_hyperparameters}
\end{table*}

\begin{table*}
\centering
\begin{tabular}{c|ccccc}
  &  MNIST / F-MNIST & CelebaHQ \\
  \hline \\[-1ex]
  
  Minibatch size & 128 & 32 \\
  \hline \\[-1ex]

  Training iterations & 50000 & 50000 \\
  \hline \\[-1ex]

  \makecell{$\lambda$ \\ (Gradient penalty \\ weight)} & 10 & 10 \\
  \hline \\[-1ex]

  \makecell{Generator \\ learning rate} & $0.001$ & $0.0005$ \\
  \hline \\[-1ex]

  \makecell{Critic \\ learning rate} & $0.0001$ & $0.0001$ \\
  \hline \\[-1ex]

  \makecell{Total \\ training time \\ (minutes)} & 250 & 1450
\end{tabular}
\caption{Hyperparameters and computational resources used for the WGAN-GP benchmark generation experiments. All unspecified hyperparameters, as well as the GPU models, are the same as for the corresponding TTC experiments in Table \ref{table:ttc_hyperparameters}. We use generator architectures matching the critic architectures, i.e. InfoGAN for MNIST / F-MNIST and SNDCGAN for CelebaHQ.}
\label{table:wgan_hyperparameters}
\end{table*}

\subsection{TTC performance versus training time}
\label{sec:perfomance_vs_time_appendix}
Figure \ref{fig:psnr_vs_time} contains a graph of the average PSNR values obtained over 200 test images with TTC at various points in training for the denoising experiments with noise standard deviations $\sigma = 0.1$, $0.15$, $0.2$. The error bars in the graph represent the standard deviations of the PSNR values over the test images. Each data point in the graph corresponds to TTC trained for denoising as reported in Table \ref{table:ttc_hyperparameters}, but with $N = 0, 1, \ \dots \ , 19$ -- in particular, each critic is trained for $2500$ iterations with minibatch size of $32$. Note that the values reported in this graph may differ slightly from those in Table \ref{table:denoisingresults}, because they were obtained over different sets of test images (recall that test images are obtained from the BSDS500 test dataset by taking $128 \times 128$ random crops). The graph shows that the improvement of TTC is marginal after around a third of the total training time, i.e. after training critics for approximately the first 10 TTC steps. The average PSNR obtained with TTC surpassed that obtained with the benchmark adversarial regularization method from \cite{lunz2018adversarial} after 3 steps for $\sigma = 0.15$ and $\sigma = 0.2$, and after 4 steps for $\sigma = 0.1$.

Figure \ref{fig:fid_vs_time} shows graphs of the FID performance plotted with respect to training time for all of the TTC and WGAN-GP generation experiments. The best FID values attained in each of these experiments are reported in Table \ref{table:TTCvsWGANGP}. The MNIST and F-MNIST experiments were run using NVIDIA P100 GPUs and the CelebaHQ experiments were run with NVIDIA V100 GPUs. We note that TTC more often than not outperforms WGAN-GP for equal training time. One exception to this is that WGAN-GP tends to do better very early on in training when only a few TTC steps have been taken. Another is that TTC and WGAN-GP perform about equally well after around 600 minutes on training on CelebaHQ. A singular feature of the FID plot for MNIST generation with TTC is a sharp but short lived increase in FID early in training -- this features stands out from the otherwise almost monotonically decreasing FID values obtained with TTC on all three datasets. This short increase in FID happened when taking the fifteenth and sixteenth TTC steps, which both used the same critic (i.e. there was no critic training at the sixteenth step), and was quickly corrected by subsequent steps. We have no clear explanation for this, but it may be due to failed training of the critic used for these steps. Aside from this anomaly, the FID values obtained with TTC are slightly more stable than those obtained with WGAN-GP for all three datasets.

\begin{figure*}[h!]
\centering
    \includegraphics[width = 0.8 \textwidth, clip = True]{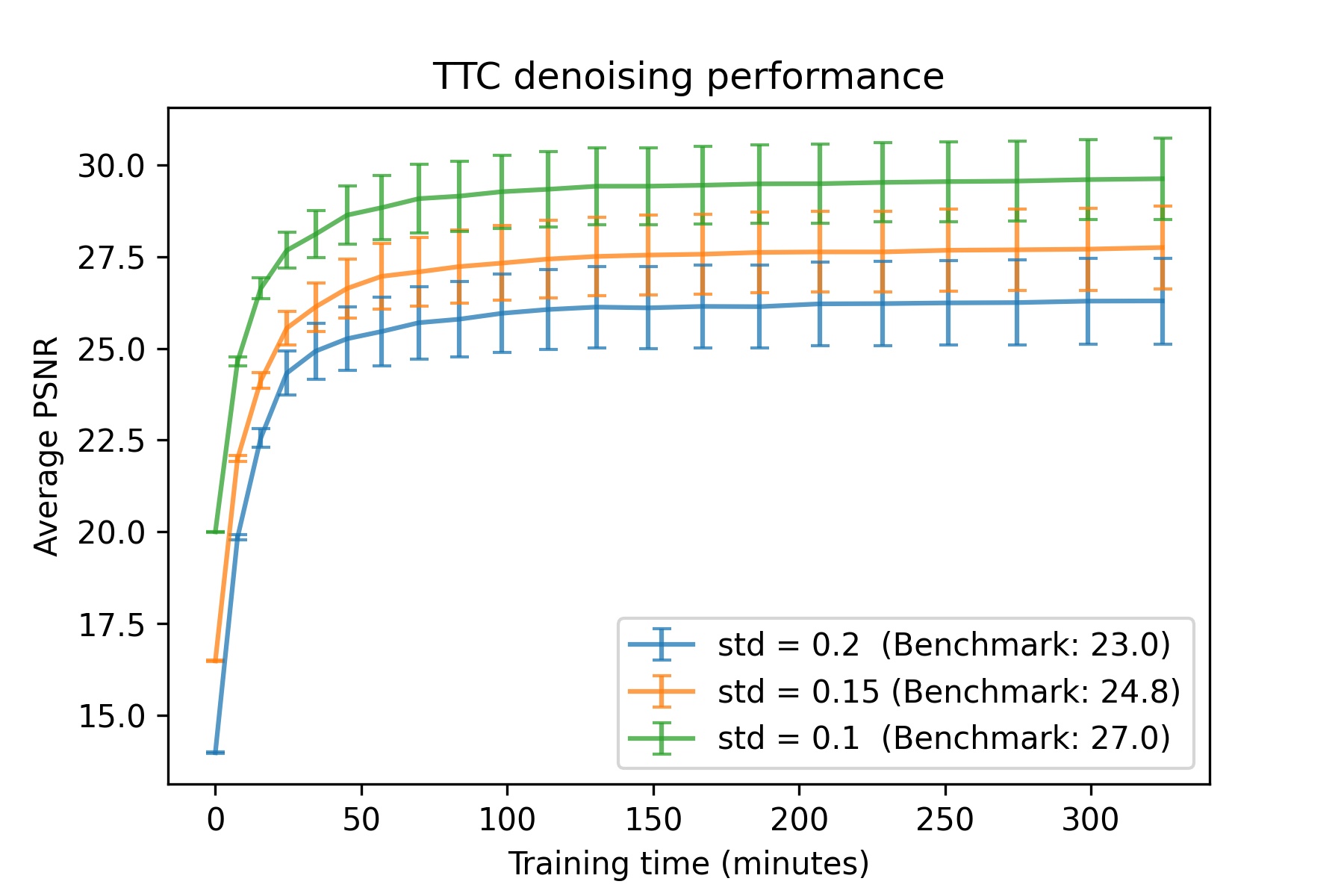}
    \caption{Average and standard deviation of the PSNR values obtained with TTC over a test dataset of 200 images, plotted against training time. The three plots correspond to the denoising experiments with noise standard deviations of $\sigma = 0.1$, $0.15$, $0.2$. The PSNR values at time $0$ correspond to the noisy images in the test dataset. The average PSNR values obtained with the adversarial regularization benchmark method are included in the legend. See the discussion in Section \ref{sec:perfomance_vs_time_appendix}.}
    \label{fig:psnr_vs_time}
\end{figure*}

\begin{figure*}[h!]
    \centering
    \includegraphics[width = 0.6 \textwidth, clip = True]{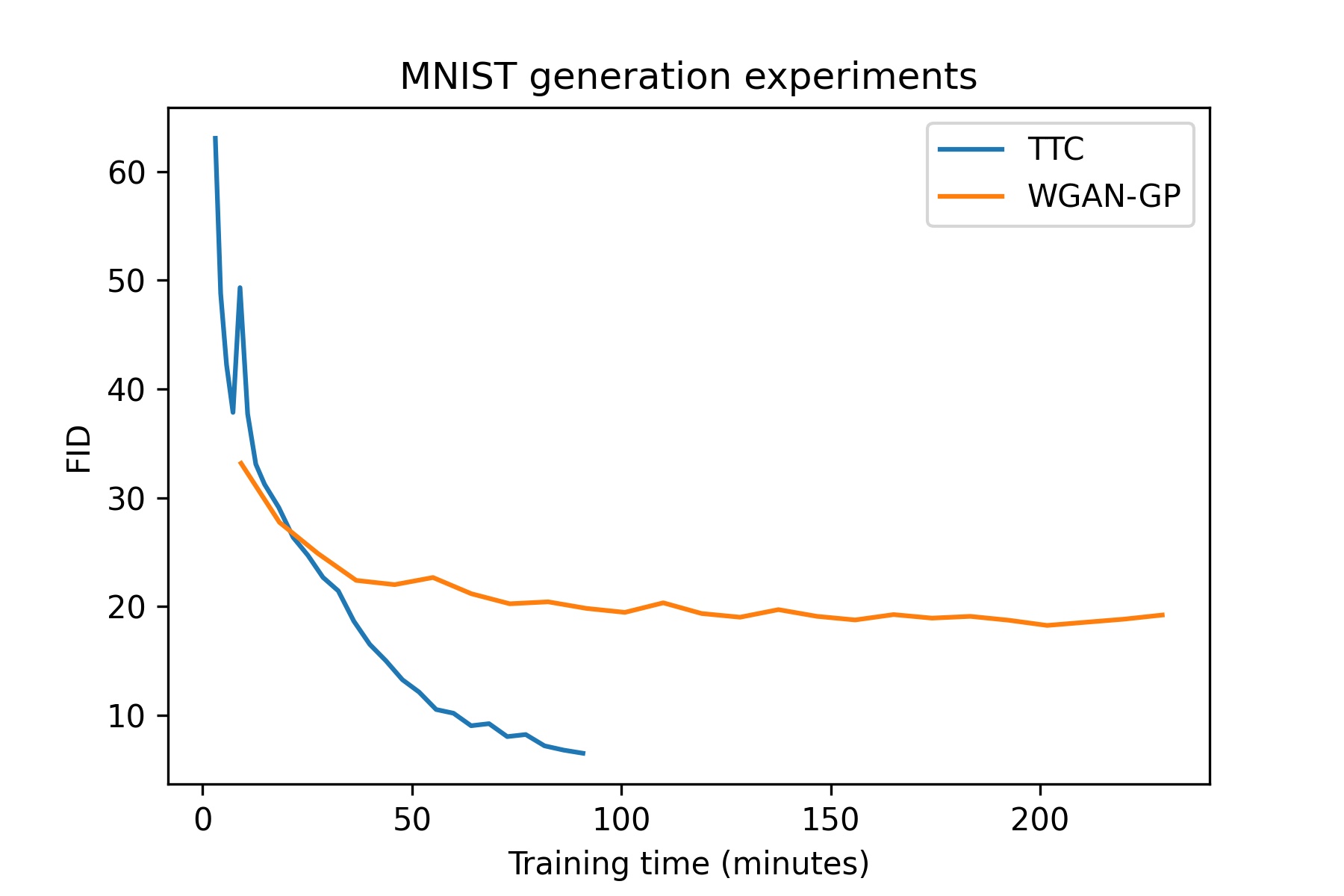}
    \includegraphics[width = 0.6 \textwidth, clip = True]{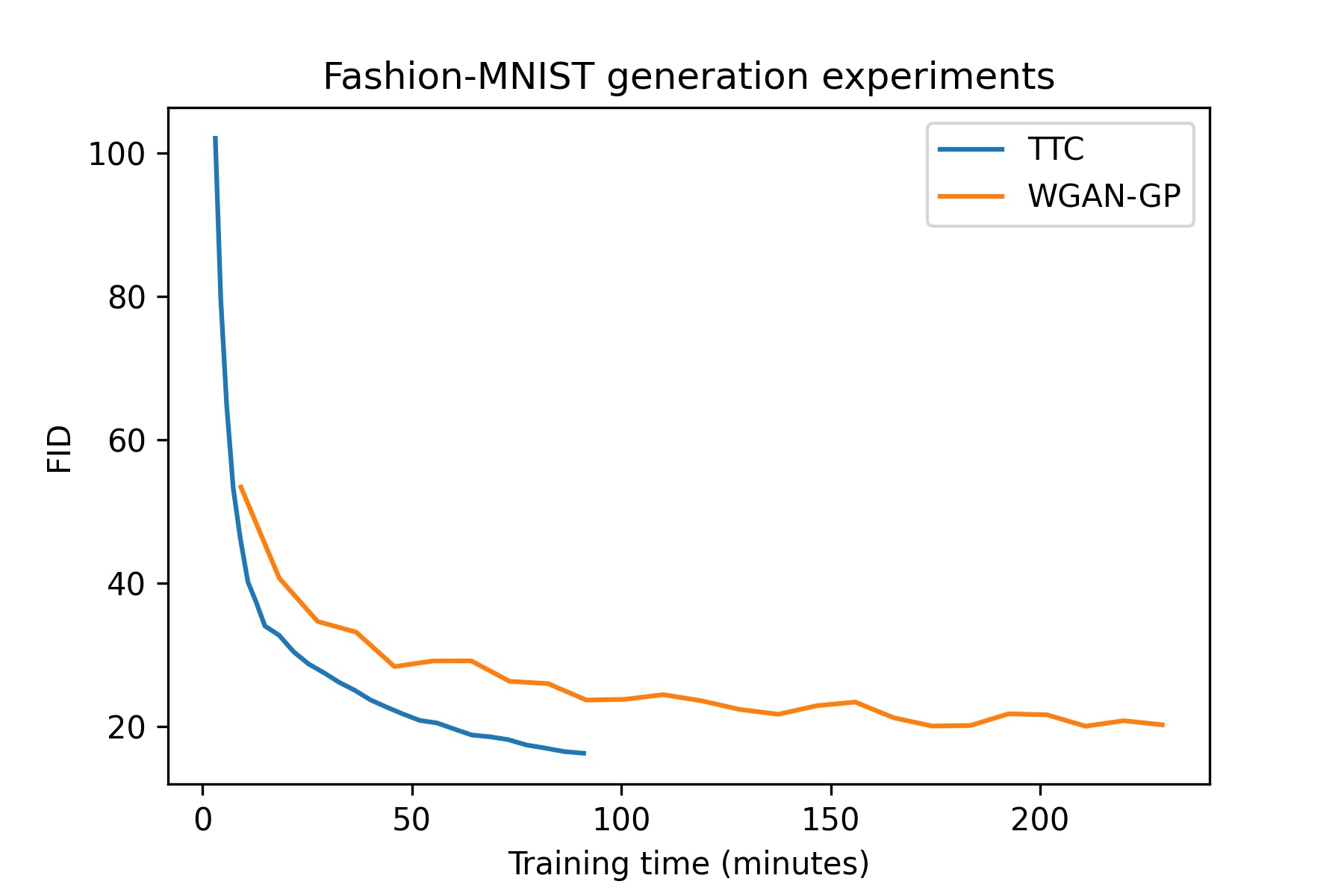}
    \includegraphics[width = 0.6 \textwidth, clip = True]{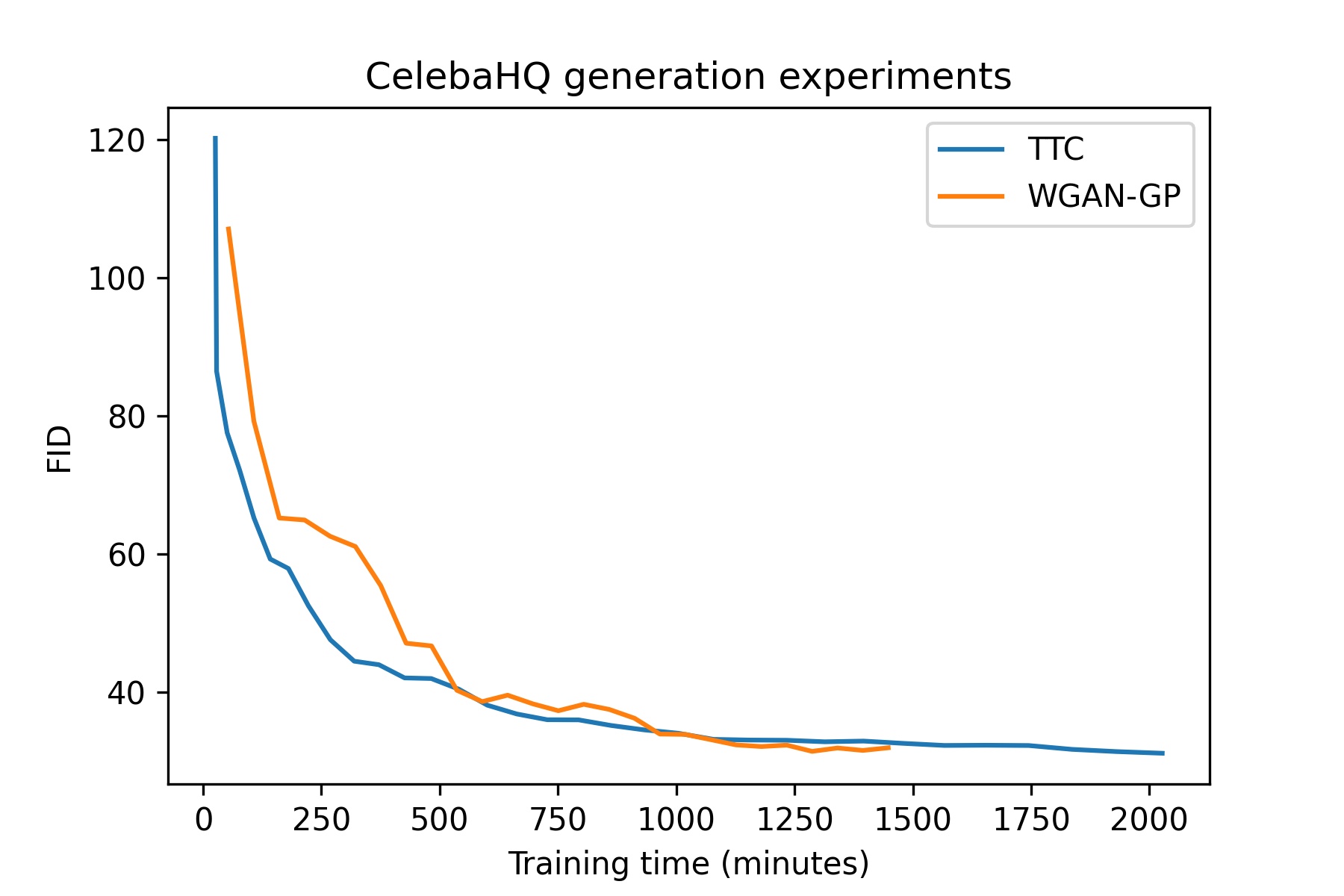}
    \caption{FID performance plotted against training time for the TTC and WGAN-GP generation experiments on all three datasets. See the discussion in Section \ref{sec:perfomance_vs_time_appendix}.}
    \label{fig:fid_vs_time}
\end{figure*}

\clearpage
\subsection{Additional denoising results}
\label{sec:denoising_appendix}

Figure \ref{fig:more_denoising} includes an enlarged version of the image from Figure \ref{fig:denoising}, as well as three additional examples of denoising comparing the benchmark method from \cite{lunz2018adversarial} to TTC. The PSNR values corresponding to these images are reported in Table \ref{table:moredenoisingresults}. The first two images in Figure \ref{fig:more_denoising} were obtained from the denoising experiments where the noise standard deviations were $\sigma = 0.1$ and $\sigma = 0.15$, respectively, whereas the last two images are both from the denoising experiment with $\sigma = 0.2$. The results of all these experiments are stated in Table \ref{table:denoisingresults}. 

\begin{figure*}[h!]
    \centering
    \includegraphics[width = 0.9 \textwidth, clip = True]{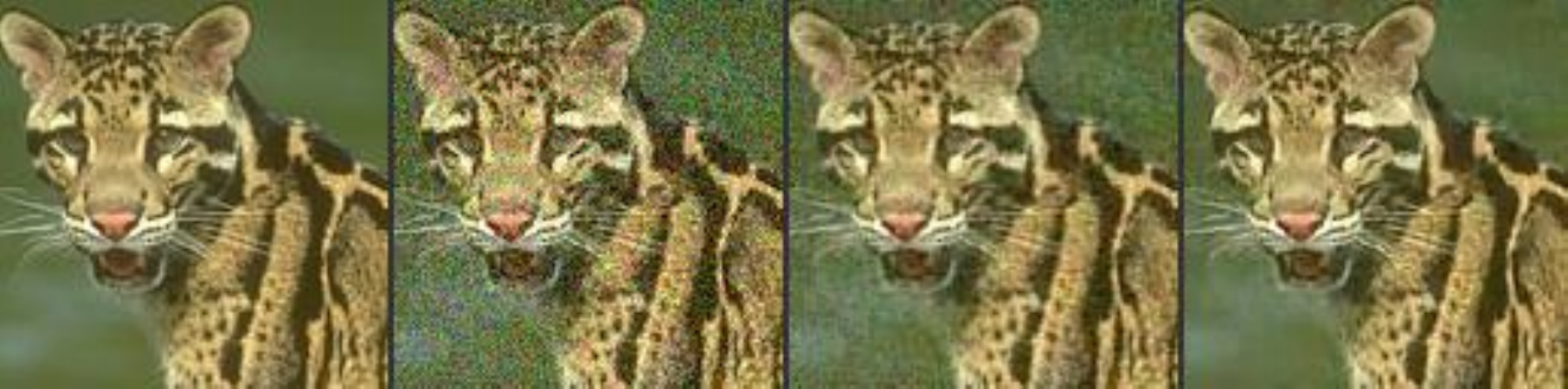} \\[0.2cm]
    \includegraphics[width = 0.9 \textwidth, clip = True]{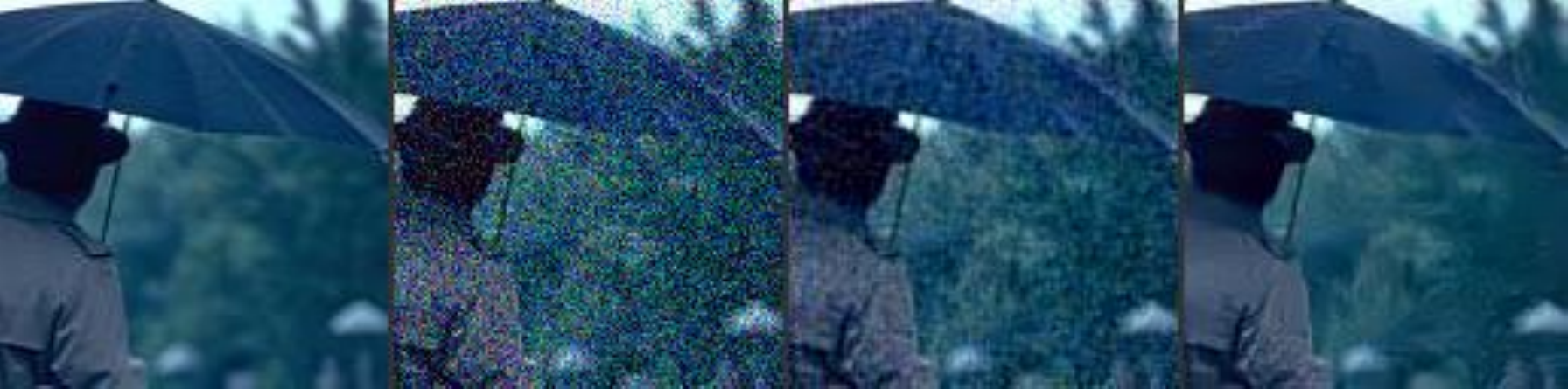} \\[0.2cm]
    \includegraphics[width = 0.9 \textwidth, clip = True]{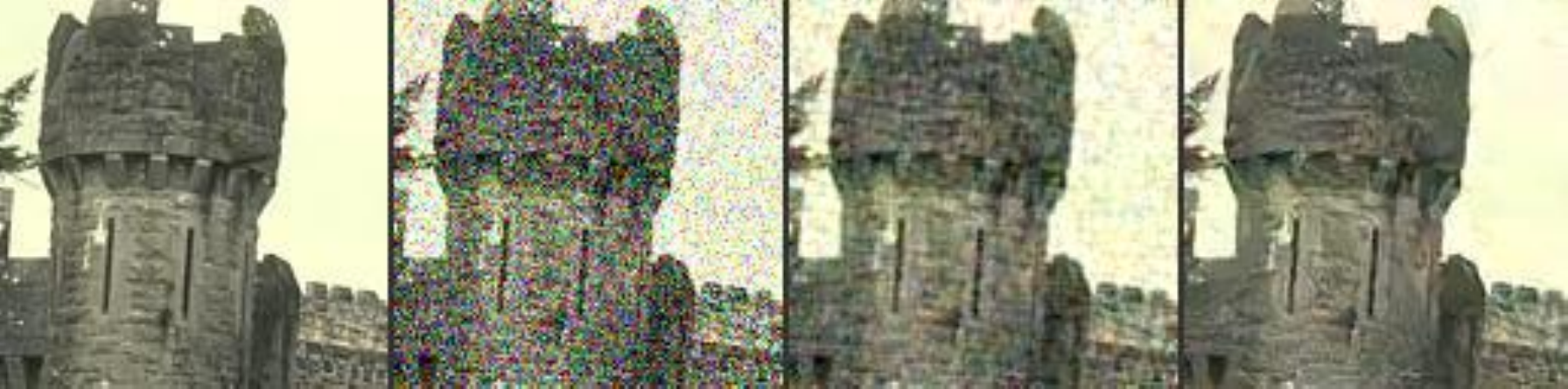} \\[0.2cm]
    \includegraphics[width = 0.9 \textwidth, clip = True]{figs/denoising_examples/denoised_reptile_large_sigma.pdf} 
    \caption{Additional denoising examples on $128 \times 128$ crops of BSDS500 test images. From left to right: original image, noisy image, image restored using the benchmark method from \cite{lunz2018adversarial}, and image restored using TTC. The top row was obtained from the denoising experiment with noise standard deviation $\sigma = 0.1$, the second rwo corresponds to $\sigma = 0.15$ and the last two rows to $\sigma = 0.2$. The PSNR values for all images are included in Table \ref{table:moredenoisingresults}.}
    \label{fig:more_denoising}
\end{figure*}

\begin{table*}[h!]
\centering
\begin{tabular}{c|cccc} 
\multicolumn{1}{l}{} &
\multicolumn{3}{c}{PSNR (dB)} \\ 
 & \makecell{$\sigma$} & Noisy Image & Adv. reg. & TTC \\ 
 \hline \\[-1ex]
Feline & $0.1$ & $20.0$ &  $26.4$ & $ \mathbf{28.1}$ \\
Umbrella & $0.15$ & $16.5$ & $25.7$ & $\mathbf{29.8}$ \\
Tower & $0.2$ &$14.0$ & $22.8$ & $\mathbf{25.1}$ \\
Reptile & $0.2$ &$14.0$ & $21.8$ & $\mathbf{23.5}$
\end{tabular}
\caption{PSNR values for the images in \Cref{fig:more_denoising}.}
\label{table:moredenoisingresults}
\end{table*}

\newpage
\subsection{Additional generated samples}
\label{sec:generation_appendix}
Figures \ref{fig:generated_mnist}, \ref{fig:generated_fashion} and \ref{fig:generated_celebahq} display uncurated generated samples from the benchmark WGAN-GP (top) and TTC (bottom) experiments described in Section \ref{sec:ttcvswgangp} for MNIST, F-MNIST and CelebaHQ, respectively. The FID values obtained during these experiments are reported in Table \ref{table:ttc_hyperparameters}.



\begin{figure*}[h]
    \centering
    \includegraphics[width = 1.0 \textwidth, clip = True]{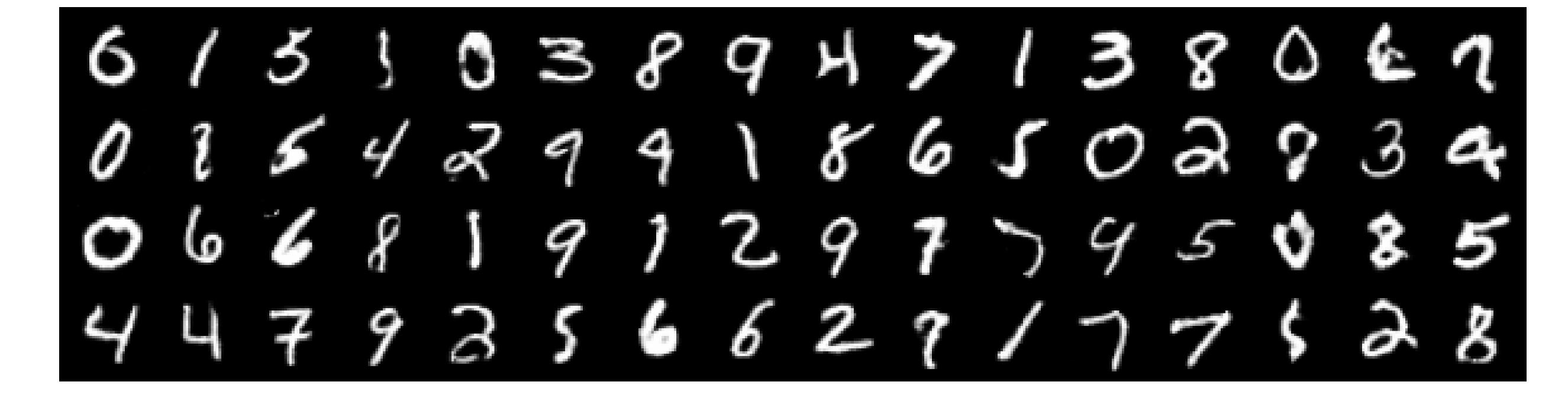} \\[0.2cm]
    \includegraphics[width = 1.0 \textwidth, clip = True]{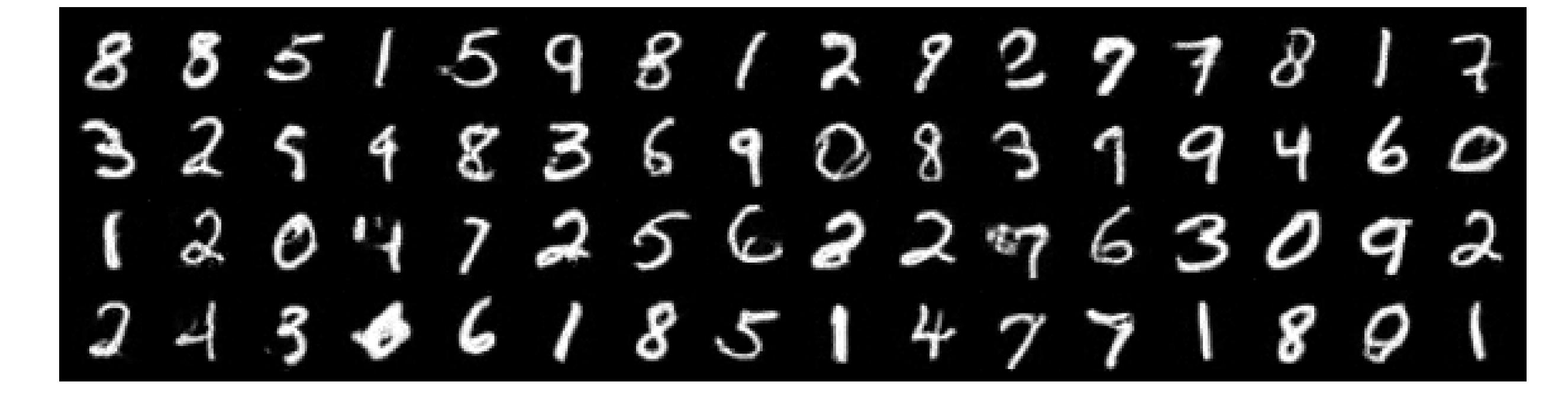} \\[0.2cm]
    \caption{Generated samples from WGAN-GP (top, FID $18.3$) and TTC (bottom, FID $6.5$) trained on MNIST as described in Section \ref{sec:ttcvswgangp}.}
    \label{fig:generated_mnist}
\end{figure*}

\begin{figure*}[h]
    \centering
    \includegraphics[width = 1.0 \textwidth, clip = True]{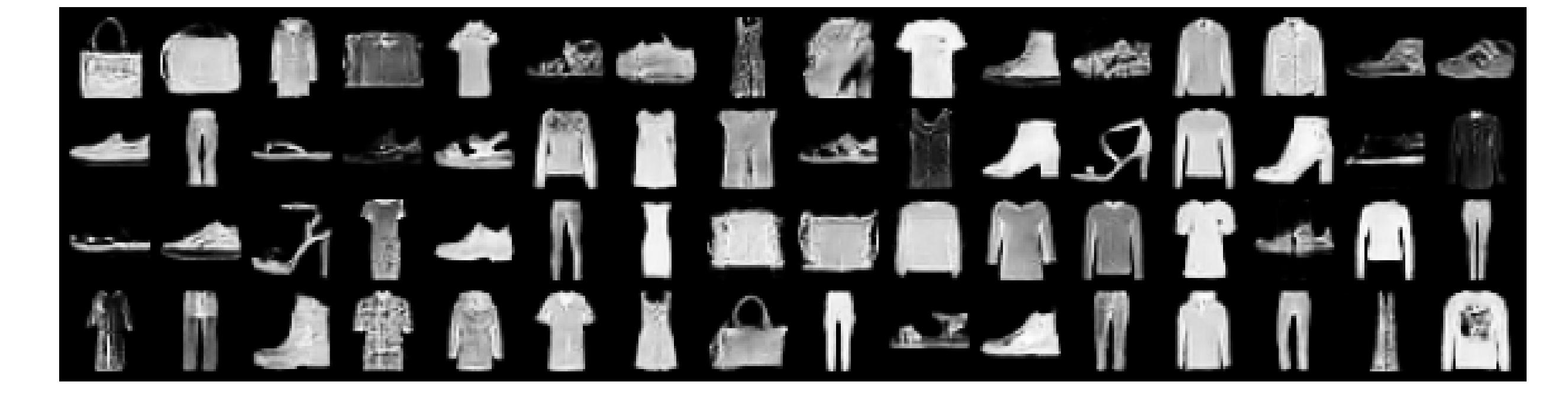}
    \includegraphics[width = 1.0 \textwidth, clip = True]{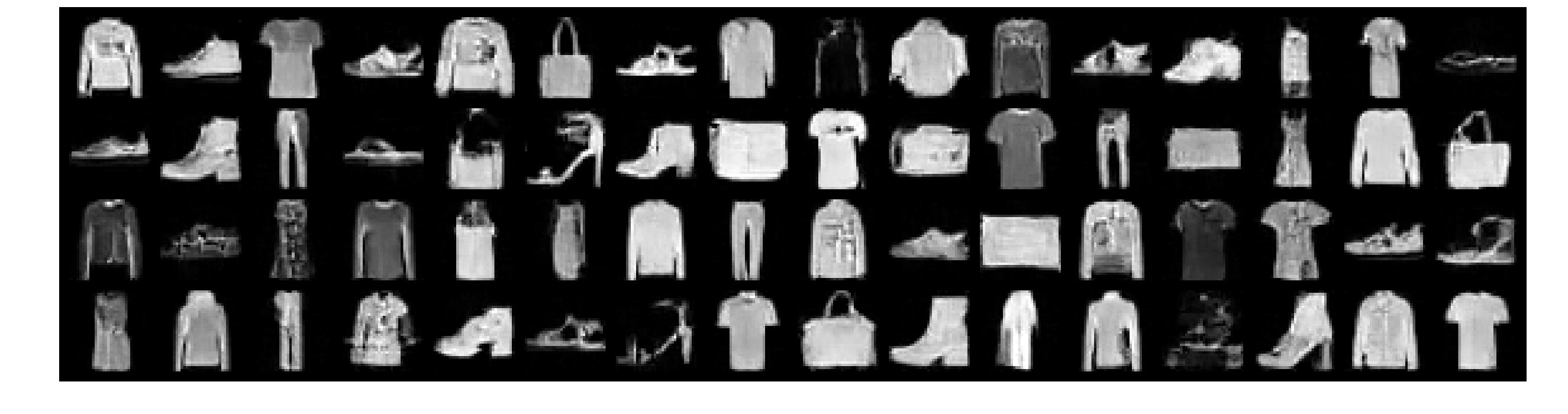} 
    \caption{Generated samples from WGAN-GP (top, FID $20.1$) and TTC (bottom, FID $16.3$) trained on Fashion-MNIST as described in Section \ref{sec:ttcvswgangp}.}
    \label{fig:generated_fashion}
\end{figure*}

\begin{figure*}[h]
    \centering
    \includegraphics[width = 1.0 \textwidth, clip = True]{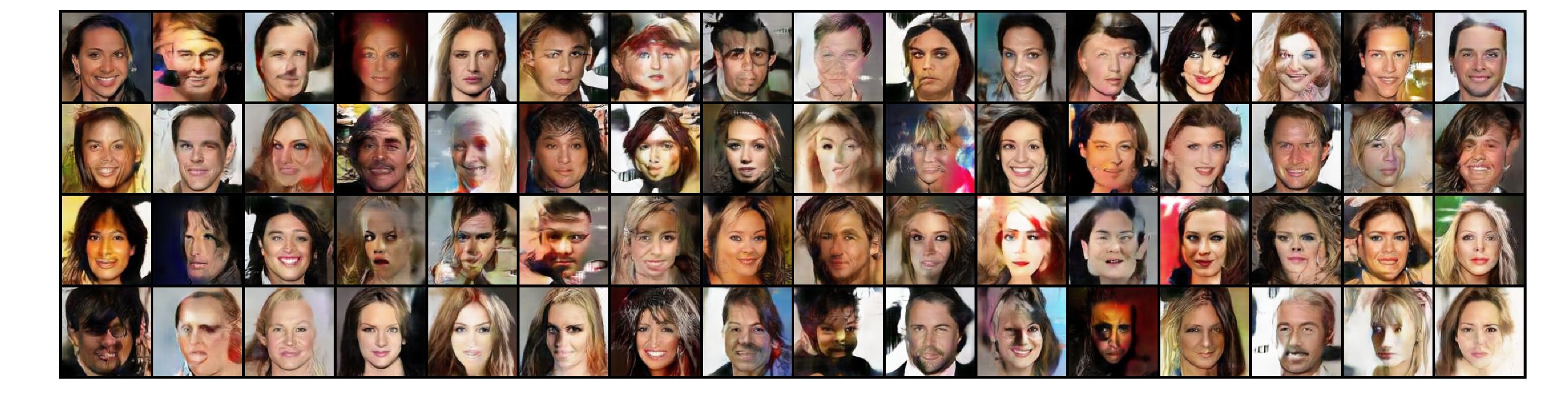} 
    \includegraphics[width = 1.0 \textwidth, clip = True]{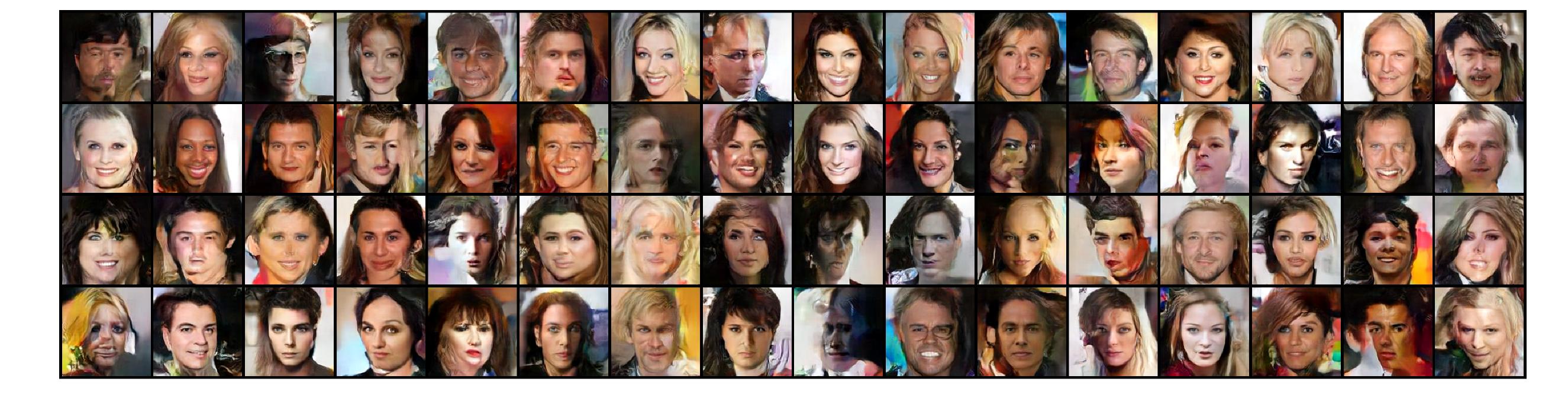} 
    \caption{Generated samples from WGAN-GP (top, FID $31.4$) and TTC (bottom, FID $31.2$) trained on CelebaHQ as described in Section \ref{sec:ttcvswgangp}.}
    \label{fig:generated_celebahq}
\end{figure*}


\clearpage
\subsection{Additional translation and deblurring images}
\label{app:translation}
Figure \ref{fig:more_deblurred} contains additional examples of deblurring with TTC. As for the image in Figure \ref{fig:deblurringexample}, the first two rows of Figure \ref{fig:more_deblurred} were obtained from a deblurring experiment where TTC was trained to reverse the effect of  a $5 \times 5$ Gaussian blurring filter with a standard deviation of $\sigma = 2$. The last two rows of Figure \ref{fig:more_deblurred} come from a deblurring experiment with a $5 \times 5$ Gaussian blurring filter witha standard deviation of $\sigma = 1$.

Figure \ref{fig:more_monets} contains additional examples of real world images being translated into Monet paintings, obtained from the TTC translation experiment described in Section \ref{sec:translation}.

\begin{figure*}
    \centering
    \includegraphics[width = 0.6 \textwidth, clip = True]{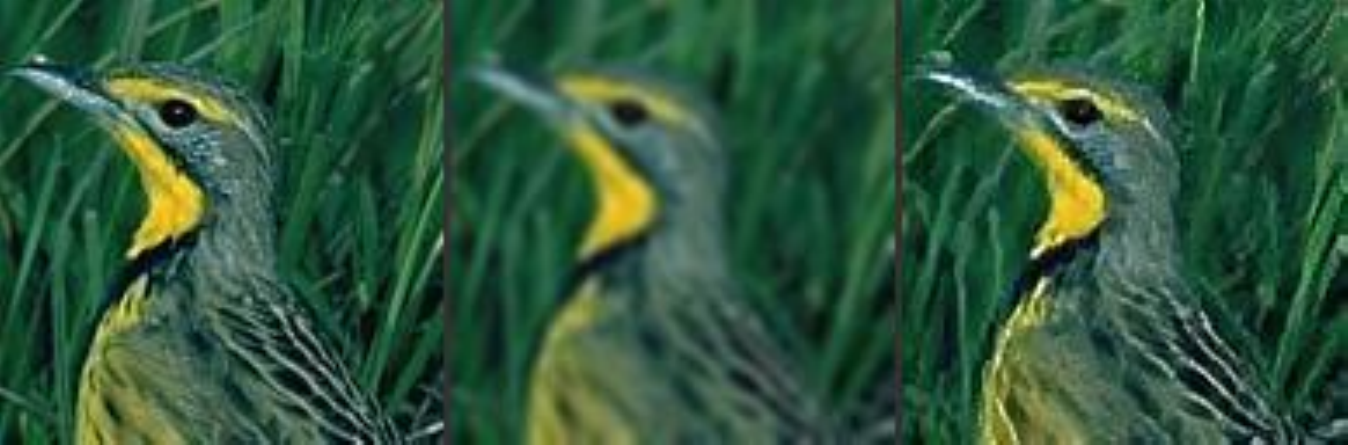} \\[0.2cm]
    \includegraphics[width = 0.6 \textwidth, clip = True]{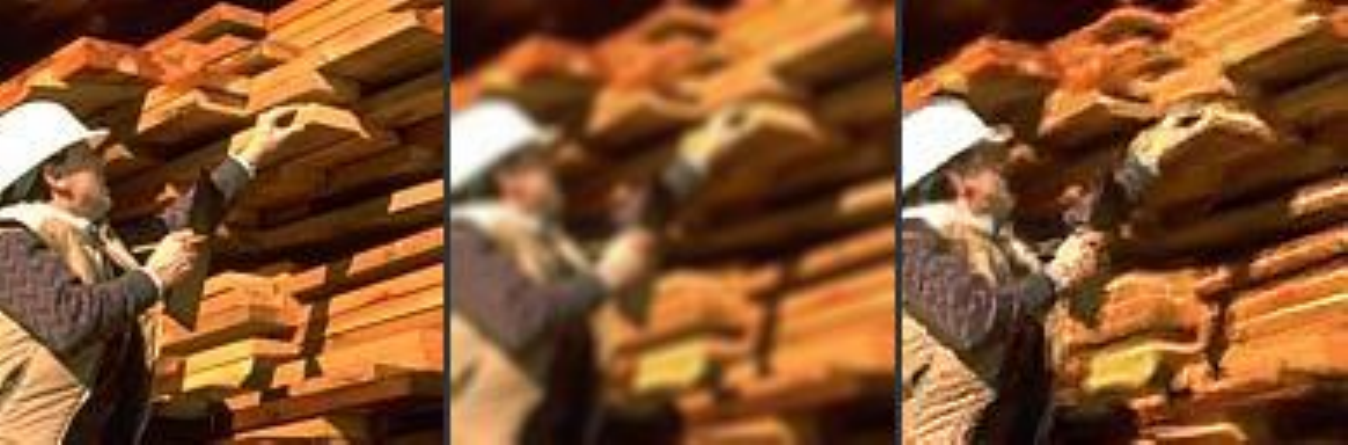} \\[0.2cm]
    \includegraphics[width = 0.6 \textwidth, clip = True]{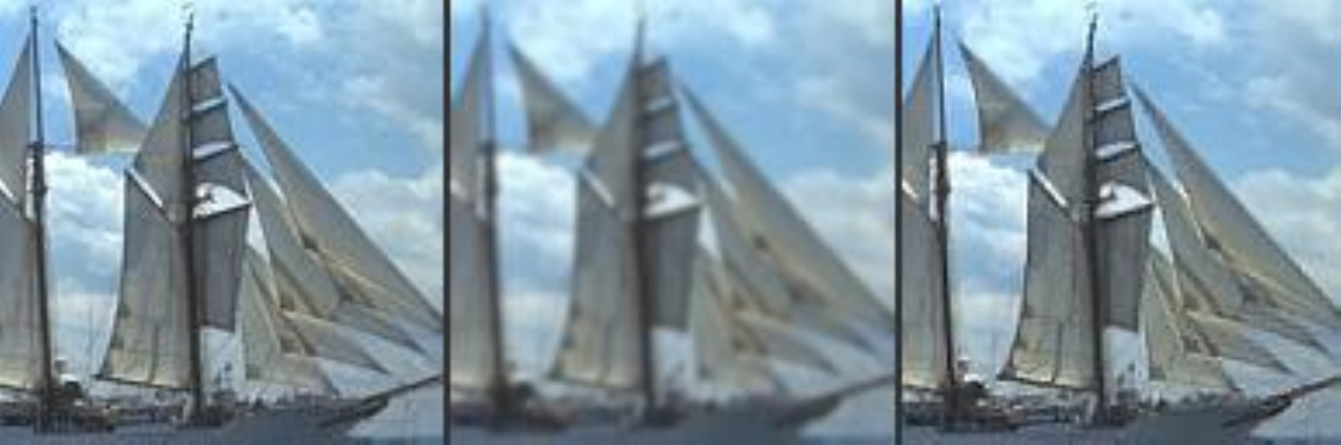} \\[0.2cm]
    \includegraphics[width = 0.6 \textwidth, clip = True]{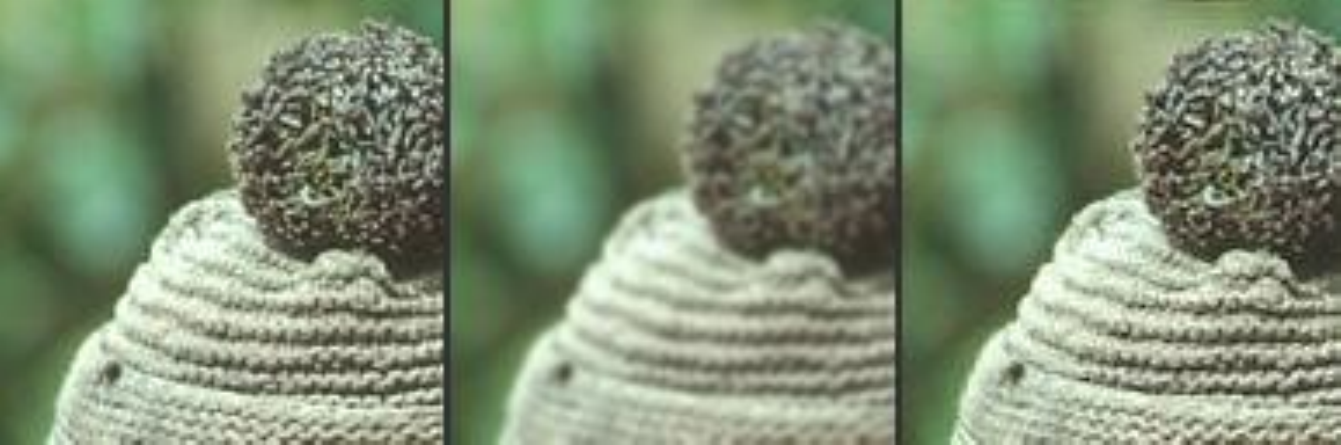}
    \caption{Additional examples of deblurring with TTC applied to $128 \times 128$ crops of BSDS500 test images. From left to right: original image, blurred image and image restored using TTC. The top two rows were obtained from an experiment where the blurring was done with a $5 \times 5$ Gaussian blurring filter with standard deviation $\sigma = 2$, whereas the bottom two rows come from an experiment with a $5 \times 5$ Gaussian blurring filter with $\sigma = 1$.}
    \label{fig:more_deblurred}
\end{figure*}

\begin{figure*}
    \centering
    \includegraphics[width = 0.6 \textwidth, clip = True]{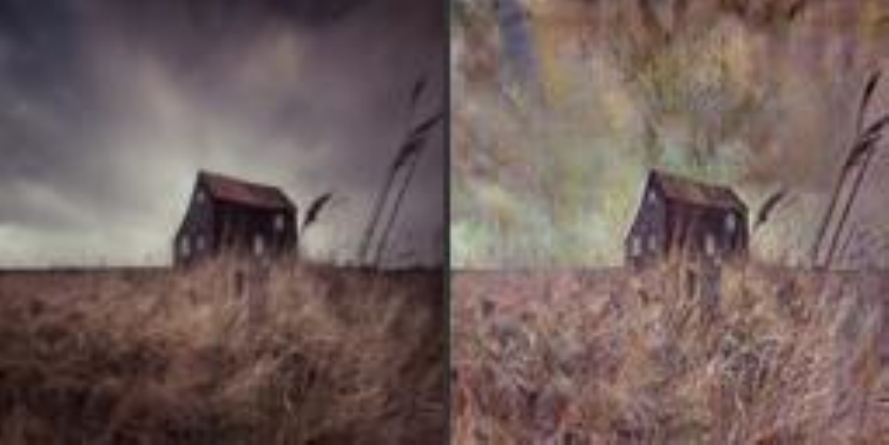} \\[0.2cm]
    \includegraphics[width = 0.6 \textwidth, clip = True]{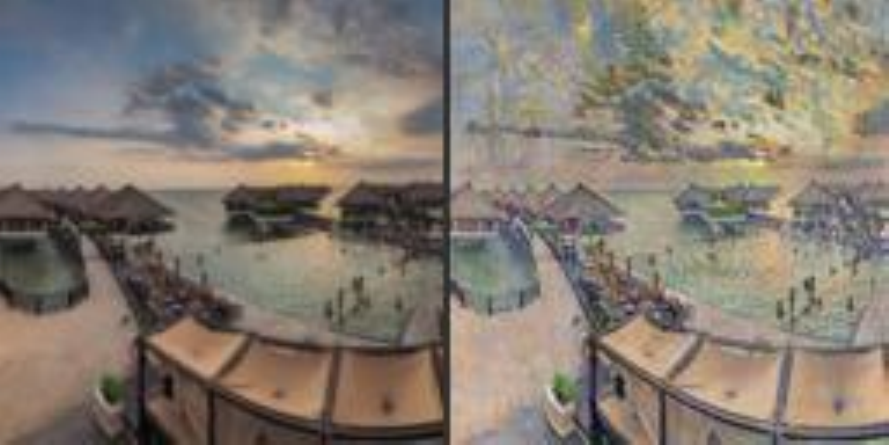} \\[0.2cm]
    \includegraphics[width = 0.6 \textwidth, clip = True]{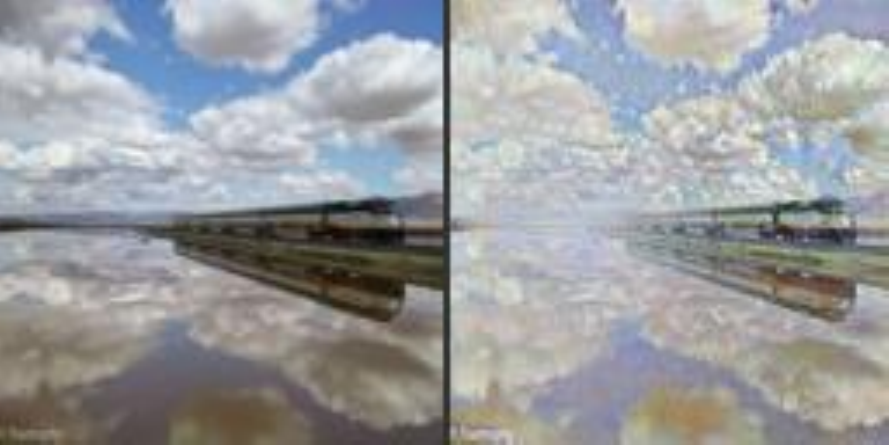} \\[0.2cm]
    \includegraphics[width = 0.6 \textwidth, clip = True]{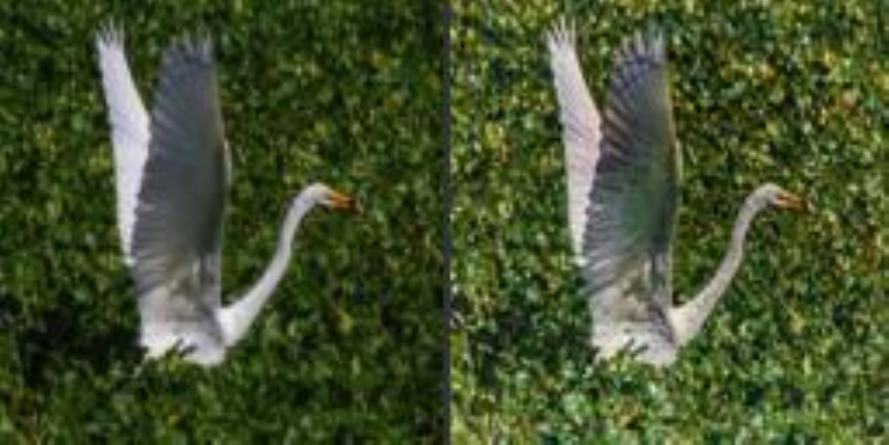}  
    \caption{Additional examples of translating landscape photos into Monet paintings using TTC. We included the fourth example as it shows that when the transport distance for a particular image is larger than the average distance (as this image is far from a typical Monet painting in the dataset), TTC may transport the image a smaller distance than it should.}
    \label{fig:more_monets}
\end{figure*}

\end{document}